\documentclass[a4paper,12pt]{article}

\usepackage[utf8]{inputenc}
\usepackage[english]{babel}
\usepackage{hyperref}
\usepackage{amsmath,amsthm}
\usepackage{enumerate}
\usepackage{enumitem} 
\usepackage{amssymb}
\usepackage{graphicx, color}
\usepackage{epsfig}
\usepackage{tikz}
\usetikzlibrary{calc,decorations.pathmorphing,decorations.text,intersections}
\usepackage{subcaption}
\usepackage{refcount}
\usepackage[hmargin=2.5cm,vmargin=3cm]{geometry}
\usepackage{mathtools, braket}
\usepackage{authblk}
\usepackage{centernot}

\theoremstyle{plain}
\newtheorem{thm}{Thm}[section]

\newtheorem{theorem}[thm]{Theorem}
\newtheorem{lemma}[thm]{Lemma}
\newtheorem{corollary}[thm]{Corollary}
\newtheorem{proposition}[thm]{Proposition}
\newtheorem{conjecture}[thm]{Conjecture}

\newtheorem{observation}[thm]{Observation}
\newtheorem{definition}[thm]{Definition}

\newenvironment{proof*}{\noindent\emph{Proof of the claim.}}{\hfill$\Diamond$}

\newtheorem{innercustomgeneric}{\customgenericname}
\providecommand{\customgenericname}{}
\newcommand{\newcustomtheorem}[2]{%
	\newenvironment{#1}[1]
	{%
		\renewcommand\customgenericname{#2}%
        \renewcommand\theinnercustomgeneric{##1}%
		\innercustomgeneric
	}
	{\endinnercustomgeneric}
}
\newcustomtheorem{customtheorem}{Theorem}
\newcustomtheorem{customlemma}{Lemma}

\renewcommand{\pod}[1]{\allowbreak\mathchoice
	{\if@display \mkern 0mu\else \mkern 0mu\fi (#1)}
	{\if@display \mkern 0mu\else \mkern 0mu\fi (#1)}
	{\mkern 1mu(\mathrm{mod}\mkern 4mu #1)}
	{\mkern 0mu(#1)}
}

\usepackage[color=green!30]{todonotes}

\usepackage{caption}
\usetikzlibrary{matrix}
\usetikzlibrary{decorations.markings}
\usetikzlibrary {arrows.meta} 
\tikzstyle{vertex}=[circle, draw, fill=black!50,
inner sep=0pt, minimum width=4pt]
\tikzset{->-/.style={decoration={
			markings,
			mark=at position .5 with {\arrow{>}}},postaction={decorate}}}

\tikzstyle{bigblue}=[color=blue, very thick, >=stealth]
\tikzstyle{lightblue}=[color=blue, thin, >=stealth]

\tikzstyle{bigred}=[color=red, very thick, >=stealth]
\tikzstyle{lightred}=[color=red, thin, >=stealth]

\tikzstyle{biggreen}=[color=black!30!green, very thick, >=stealth]
\tikzstyle{lightgreen}=[color=black!30!green,  thin, >=stealth]

\pgfdeclarelayer{bg}    
\pgfsetlayers{bg,main}  

\begin{document}

	\title{\bf {\Large Critically $3$-frustrated signed graphs}}
	\author[1]{Chiara Cappello}
	\author[2]{Reza Naserasr}
	\author[1]{Eckhard Steffen}
	\author[3]{Zhouningxin Wang}
	\affil[1]{\small Paderborn University, Department of Mathematics, Warburger Str. 100, 33098 Paderborn, Germany. }
	\affil[2]{\small Université Paris Cité, CNRS, IRIF, F-75013, Paris, France.} 
	\affil[3]{\small School of Mathematical Sciences and LPMC, Nankai University, Tianjin 300071, China.\linebreak Emails: 
    \{ccappello, es\}@mail.uni-paderborn.de, reza@irif.fr,   wangzhou@nankai.edu.cn} 
    \date{}
	\maketitle
	\renewcommand{\baselinestretch}{1.2}
	\linespread{1.2}
	\begin{abstract}
	Extending the notion of maxcut, the study of the frustration index of signed graphs is one of the basic questions in the theory of signed graphs. Recently two of the authors initiated the study of critically frustrated signed graphs. That is a signed graph whose frustration index decreases with the removal of any edge. The main focus of this study is on critical signed graphs which are not edge-disjoint unions of critically frustrated signed graphs (namely non-decomposable signed graphs) and which are not built from other critically frustrated signed graphs by subdivision. We conjecture that for any given $k$ there are only finitely many critically $k$-frustrated signed graphs of this kind. 
        	
    Providing support for this conjecture we show that there are only two of such critically $3$-frustrated signed graphs where there is no pair of edge-disjoint negative cycles. Similarly, we show that there are exactly ten critically $3$-frustrated signed planar graphs that are neither decomposable nor subdivisions of other critically frustrated signed graphs.
    We present a method for building  non-decomposable critically frustrated signed graphs based on two given such signed graphs. We also show that the condition of being non-decomposable is necessary for our conjecture.
	\end{abstract}

	\section{Introduction}
	In this paper, graphs are allowed to have multiedges and loops. For a graph $G$, let $E(G)$ and $V(G)$ denote the set of edges and the set of vertices of $G$, respectively. A \emph{signed graph} $(G, \sigma)$ is a graph $G$ together with an assignment $\sigma: E(G) \rightarrow \{+, -\}$ called \emph{signature}, where $\{+, -\}$ is viewed as a multiplicative group. An edge $e$ of $(G, \sigma)$ is called \emph{negative} if $\sigma(e)=-$ and \emph{positive} otherwise. The set of negative edges of $(G, \sigma)$ is denoted by $E^-_\sigma(G)$. We may simply write $E^-_\sigma$ if the underlying graph is clear from the context. Furthermore, if $E^-_\sigma(G) = \emptyset$, then such a signature $\sigma$ is called the \emph{all-positive} signature, and the corresponding signed graph is denoted by $(G, +)$. Similarly, if $E^-_\sigma(G) = E(G)$, then $\sigma$ is called the \emph{all-negative} signature, and the corresponding signed graph is denoted by $(G, -)$. For a signed graph $(G, \sigma)$ and a subgraph $H$ of $G$, with rather an abuse of notation, we write $(H, \sigma)$ to denote the signed graph $(H, \sigma |_{_{E(H)}})$.
	
	Let $G$ be a graph. For a vertex $v \in V(G)$, the \emph{degree} of $v$, denoted $d_{G}(v)$, or simply $d(v)$ when $G$ is clear from the context, is the number of edges incident with $v$, where loops are counted twice. For $X \subseteq V$, we denote by $\partial_{G}(X)$ an \emph{edge-cut} in $G$, which is a set of edges defined as follows: $\partial_{G}(X):= \{xy \in E(G) : x \in X, y \notin X\}$. The cardinality of an edge-cut $\partial_{G}(X)$ is denoted by $d_{G}(X)$, i.e., $d_{G}(X)=|\partial_{G}(X)|$. In particular, when we work with signed graphs without loops, if $X=\{v\}$, then in place of $d_G(\{v\})$ we write $d_G(v)$. Given a signature $\sigma$ on $G$, a refinement of the notation $d_{G}(X)$ for $(G, \sigma)$ is defined as follows: $$d^-_{(G, \sigma)}(X):= |\partial_{(G, \sigma)}(X) \cap E^-_\sigma|\text{~~and~~}d^+_{(G, \sigma)}(X):= | \partial_{(G, \sigma)}(X) \setminus E^-_\sigma|.$$ An edge-cut $\partial_{G}(X)$  is said to be \emph{equilibrated} under $\sigma$ if $d^+_{(G, \sigma)}(X)=d^-_{(G, \sigma)}(X)$.
	Whenever it is clear from the context, we may omit the index ``$(G, \sigma)$'' of the notations introduced above.

    A \emph{cycle} of $G$ is a connected $2$-regular subgraph. A cycle in $(G, \sigma)$ is said to be \emph{positive} (respectively, \emph{negative}) if it contains an even (respectively, odd) number of negative edges. A signed graph $(G, \sigma)$ is \emph{balanced} if it contains no negative cycle and \emph{unbalanced} otherwise.
	
	\smallskip
	For a signed graph $(G, \sigma)$, \emph{switching} at a vertex $v \in V(G)$ is to multiply the signs of all edges incident with $v$ by $-$. For an edge-cut $\partial (X)$, switching at $\partial (X)$ is to switch at all vertices in $X$. 
	Furthermore, two signed graphs $(G, \sigma)$ and $(G, \sigma ')$ are \emph{switching equivalent} if one can be obtained from the other by a series of switchings at some vertices. In such a case we may also say $\sigma$ is switching equivalent to $\sigma'$. It has been proved by Zaslavsky in \cite{Z82} that two signatures on the same graph are switching equivalent if and only if they induce the same set of negative cycles.
	
	Balanceness is not a common state for signed graphs, hence it is interesting to define parameters in order to compute how far a signed graph is from being balanced, see for example~\cite{AW19} and references therein. One of the basic parameters to measure this is the \emph{frustration index} of a signed graph $(G, \sigma)$, denoted by $\ell(G, \sigma)$, which is defined as follows: $$\ell(G, \sigma) = \min\limits_{\pi} \{|E^-_{\pi}|: (G, \pi) \text{~is switching equivalent to~} (G, \sigma)\}.$$ If $\ell(G, \sigma)=k$, then $(G, \sigma)$ is said to be \emph{$k$-frustrated}.
	For a signed graph $(G, \sigma)$, a signature $\pi$ is said to be a \emph{minimal equivalent signature} of $(G, \sigma)$ if $(G, \pi)$ is switching equivalent to $(G, \sigma)$ and there is no equivalent signature $\pi'$ such that  $E^-_{\pi'}\subset E^-_{\pi}$. In particular, a \emph{minimum equivalent signature}, or simply a \emph{minimum signature}, of $(G, \sigma)$ is a signature $\pi$ such that $(G, \pi)$ is switching equivalent to $(G, \sigma)$ and $|E^-_\pi|= \ell(G, \sigma)$. 

    \smallskip
	A closely related parameter to measure the balanceness of a signed graph is based on the notion of \emph{negative-cycle cover}: That is a set of edges that contains at least one edge of each negative cycle of $(G, \sigma)$. For a signed graph, that the order of a minimum negative-cycle cover and the frustration index are the same is a consequence of the following folklore lemma for which we provide a simple proof.
	
	\begin{lemma}\label{lem:NegativeCycleCover}
		There is a one-to-one correspondence between the set of minimal negative-cycle covers of a signed graph $(G, \sigma)$ and minimal equivalent signatures of $(G,\sigma)$.
	\end{lemma}
	
	\begin{proof}
		First, note that any signature, in particular a minimal signature, is a negative-cycle cover.
		Let $E'$ be a minimal negative-cycle cover. We claim that $E'$ is the set of negative edges of a minimal equivalent signature of $(G, \sigma)$. To this end, observe that $(G-E', \sigma)$ is balanced and thus, it can be changed to $(G-E', +)$ by switching at a set $X$ of vertices.
		Then, after switching at $X$, the set of negative edges of the resulting signed graph has to be $E'$. Otherwise, the set of negative edges is a proper subset of $E'$ and also a  negative-cycle cover, contradicting the minimality of $E'$.
		This also implies that a negative-cycle cover provided by a minimal equivalent signature is minimal, as otherwise, an included minimal negative-cycle cover would be a smaller signature.
	\end{proof}

	\subsection[Critically frustrated signed graphs]{Critically $k$-frustrated signed graphs}

	As computing the frustration index of a signed graph $(G, -)$ is equivalent to computing the size of a maximum cut of $G$, the problem of computing $\ell(G, \sigma)$ for an input signed graph $(G,\sigma)$ is an NP-hard problem \cite{AM20}. This motivates the study of the structure of signed graphs with high frustration index. For example, a basic observation is that the existence of $k$ edge-disjoint negative cycle in $(G, \sigma)$ implies $\ell(G, \sigma)\geq k$. To better understand the structural properties of signed graphs with the frustration index being at least $k$, the notion of critically $k$-frustrated signed graphs is introduced in \cite{CS22}. This notion is formally defined as follows. 
	
	\begin{definition}
		A signed graph $(G, \sigma)$ is \emph{critically $k$-frustrated} if $\ell(G, \sigma) = k$ and for each edge $e \in E(G)$, we have $\ell(G-e, \sigma)= k-1$.
	\end{definition}
	
	We note that critically $k$-frustrated signed graphs can be characterized in the following way.
	
	\begin{theorem}{\rm \cite{CS22}}\label{thm:characterizations}
		Let $k$ be a positive integer and $(G,\sigma)$ be a $k$-frustrated signed graph. The following statements are equivalent. 
		\begin{enumerate}[label=(\arabic*)]
			\item $(G,\sigma)$ is critically $k$-frustrated.
			\item For each edge $e \in E(G)$, there exists a minimum signature $\sigma'$ of $(G, \sigma)$ such that $\sigma'(e)=-$.
			\item If $|E^-_{\sigma}|=\ell(G, \sigma)$, then every positive edge of $(G, \sigma)$ is contained in an equilibrated edge-cut under $\sigma$. 
		\end{enumerate}
	\end{theorem}

	Note that, given a critically $k$-frustrated signed graph $(G, \sigma)$ with $\sigma$ being a minimum signature, for each edge-cut $\partial(X)$ it holds that $d^-(X) \leq d^+(X)$.

	\smallskip
	Given positive integers $k, k_1, \ldots, k_t$ such that $k=\sum^{t}_{i=1}k_i$, a critically $k$-frustrated signed graph $(G, \sigma)$ is said to be \emph{$(k_1, \dots,k_t)$-decomposable} if $E(G)$ can be partitioned into $t$ parts $E_1, E_2, \ldots, E_t$ such that for each $i$, $i\in \{1,2, \ldots, t\}$, the signed subgraph $(G[E_i], \sigma)$ is critically $k_i$-frustrated. If $(G, \sigma)$ is $(k_1, \dots,k_t)$-decomposable for some $t\geq 2$, then we simply say it is decomposable. A critically frustrated signed graph that is not decomposable is said to be \emph{non-decomposable}.

	\begin{observation}\label{obs:cycleunion}
		Let $(G, \sigma)$ be a critically $k$-frustrated signed graph. If for $k_1, \dots, k_t$ with $k=\sum^n_{i=1}k_i$ we find edge-disjoint signed subgraphs $(G_i, \sigma)$ of $(G, \sigma)$ such that $(G_i, \sigma)$ is $k_i$-frustrated, then $E(G)=\bigcup_{i\in [t]}E(G_i)$ and thus $(G, \sigma)$ is $(k_1, \dots,k_t)$-decomposable. 
  
        In particular, a critically $k$-frustrated signed graph containing $k$ edge-disjoint negative cycles is the union of all these negative cycles.
	\end{observation}
 
	Note also that, if a critically $k$-frustrated signed graph $(G, \sigma)$ contains two parallel edges $e_1$ and $e_2$ having different signs, then each equilibrated cut of $(G, \sigma)$ is also an equilibrated cut of $(G-\{e_1,e_2\}, \sigma)$.
	Hence, by Theorem~\ref{thm:characterizations}, the following observation holds.
	\begin{observation}\label{obs:loops}
		Let $(G, \sigma)$ be a critically $k$-frustrated signed graph. If $(G, \sigma)$ contains a loop, then the loop is negative and $(G, \sigma)$ is decomposable. 
		If $(G, \sigma)$ contains two parallel edges of different signs, then $(G, \sigma)$ is decomposable.
	\end{observation}
	
    Since a decomposable signed graph relies on the structures of its critical subgraphs, in the following, we choose to focus on critically frustrated signed graphs which are not decomposable. In particular, from here on, the signed graphs that we consider have no loop and no parallel edges of different signs.

    Another graph operation to preserve the property of being critically frustrated is the following (modified) notion of \emph{subdivision} as introduced in~\cite{CS22}.
	For a signed graph $(G, \sigma)$ and a positive integer $t$, a \emph{$t$-multiedge} between two vertices $x, y$ of $G$ is a set of $t$ edges connecting $x$ and $y$, denoted by $E_{xy}$.  
	As we have assumed (from here on) that there are no parallel edges of different signs, all the edges of a $t$-multiedge $E_{xy}$ are of the same sign and, depending on this sign, it will be referred to as \emph{all-positive} or \emph{all-negative} $t$-multiedge.
	
	Given a signed graph $(G, \sigma)$ and an all-positive (resp. all-negative) $t$-multiedge $E_{xy}$ of $(G, \sigma)$, let $(G', \sigma')$ denote the signed graph obtained from $(G-E_{xy}, \sigma)$ by adding a new vertex $v$, and adding two $t$-multiedges $E_{xv}$ and $E_{vy}$, so that $E_{xv}$ is all-positive (resp. all-negative) and $E_{vy}$ is all-positive (in both cases). We say that $(G', \sigma ')$ is obtained from $(G, \sigma)$ by \emph{subdividing} at a multiedge $E_{xy}$. If a signed graph $(H, \pi)$ is obtained by subdividing at a series of multiedges of $(G, \sigma)$, then we say that $(H, \pi)$ is a \emph{subdivision} of $(G, \sigma)$. 
	We note that the equivalence class of $(H, \pi)$ is uniquely determined by $(G, \sigma)$. If a signed graph $(G, \sigma)$ is not a proper subdivision of any signed graph, then we say that $(G, \sigma)$ is \emph{irreducible}. 
	Note that, if $(G, \sigma)$ is decomposable, then all its subdivisions are also decomposable. See Figure \ref{fig:subd} for examples.

	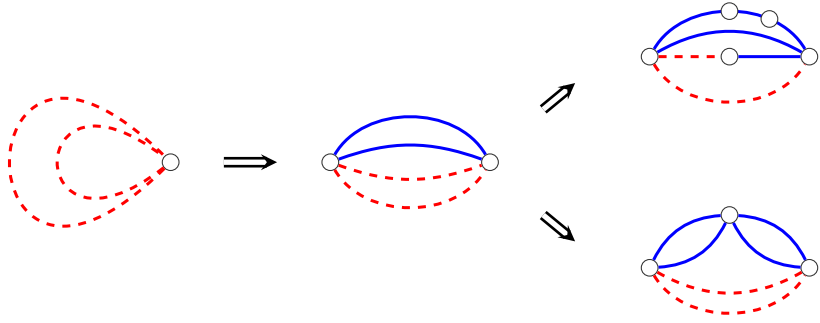
\begin{figure}[htbp]
		\centering
		\begin{tikzpicture}[scale=.7]
        \centering
	   \draw[rotate=0] (0,0)  node[circle,draw=black!80, inner sep=0mm, minimum size=2.2mm] (x){};
			\draw[rotate=0] [line width=0.4mm, dashed, red] (x) .. controls (-2.8,-2.2) and (-2.8,2.2)  .. (x);	
			\draw[rotate=0] [line width=0.4mm, dashed, red] (x) .. controls (-4,-4) and (-4,4)  .. (x);
			
			\draw[>=stealth, thick, line width=1pt, double distance=2pt, ->] (1,0) -- (2,0);
			
			\draw[rotate=0] (3,0)  node[circle,draw=black!80, inner sep=0mm, minimum size=2.2mm] (x){};
			\draw[rotate=0] (6,0)  node[circle,draw=black!80, inner sep=0mm, minimum size=2.2mm] (y){};
			\draw [bend left=20,  line width=0.4mm, blue] (x) to (y);
			\draw [bend right=20, dashed, line width=0.4mm, red] (x) to (y);
			\draw [bend left=60,  line width=0.4mm, blue] (x) to (y);
			\draw [bend right=60, dashed, line width=0.4mm, red] (x) to (y);
			
			\draw[>=stealth, thick, line width=1pt, double distance=2pt, ->] (7,1) -- (7.6,1.5);
			\draw[>=stealth, thick, line width=1pt, double distance=2pt, ->] (7,-1) -- (7.6,-1.5);
			
			\draw[rotate=0] (9,2)  node[circle,draw=black!80, inner sep=0mm, minimum size=2.2mm] (x1){};
			\draw[rotate=0] (10.5,2)  node[circle,draw=black!80, inner sep=0mm, minimum size=2.2mm] (x2){};
			\draw[rotate=0] (12,2)  node[circle,draw=black!80, inner sep=0mm, minimum size=2.2mm] (x3){};
			
			\draw [bend right=60, dashed, line width=0.4mm, red] (x1) to (x3);
			\draw [dashed, line width=0.4mm, red] (x1) to (x2);	
			\draw [line width=0.4mm, blue] (x2) to (x3);
			\draw [bend left=30, line width=0.4mm, blue] (x1) to (x3);
			\draw [bend left=60, line width=0.4mm, blue, name path=upper_positive] (x1) to (x3);
			
			\path[name path=h1] (10.50,2)-- (10.499,4);
			\path[name path=h2] (11.25,2)-- (11.25,4);
			\path[name intersections={of=h1 and upper_positive,by={i1}}];
			\path[name intersections={of=h2 and upper_positive,by={i2}}];
			
			\draw[rotate=0] (i1)  node[circle,draw=black!80, inner sep=0mm, minimum size=2.2mm, fill=white] (x4){};
			\draw[rotate=0] (i2)  node[circle,draw=black!80, inner sep=0mm, minimum size=2.2mm, fill=white] (x5){};
			
			\draw[rotate=0] (9,-2)  node[circle,draw=black!80, inner sep=0mm, minimum size=2.2mm] (x1){};
			\draw[rotate=0] (12,-2)  node[circle,draw=black!80, inner sep=0mm, minimum size=2.2mm] (x3){};
			\draw[rotate=0] (10.5,-1)  node[circle,draw=black!80, inner sep=0mm, minimum size=2.2mm] (x2){};
			\draw [bend right=60, dashed, line width=0.4mm, red] (x1) to (x3);
			\draw [bend right=30, dashed, line width=0.4mm, red] (x1) to (x3);
			\draw [bend left=30, line width=0.4mm, blue] (x1) to (x2);
			\draw [bend right=30, line width=0.4mm, blue] (x1) to (x2);
			\draw [bend left=30, line width=0.4mm, blue] (x3) to (x2);
			\draw [bend right=30, line width=0.4mm, blue] (x3) to (x2);
			
		\end{tikzpicture}
		\caption{An example of possible subdivisions.}
		\label{fig:subd}
	\end{figure}
	
	The importance of this generalized notion of the subdivision in the study of critically frustrated signed graphs is highlighted in the following proposition.
	
	\begin{proposition}{\rm \cite{CS22}}
		For a signed graph $(G, \sigma)$ and a subdivision $(G', \sigma')$ of it, we have: (i). $\ell(G,\sigma)=\ell(G',\sigma')$; (ii). $(G,\sigma)$ is critically frustrated if and only if $(G', \sigma')$ is critically frustrated.
	\end{proposition}
	
	Hence, without loss of generality, we can always limit our study to the class of irreducible signed graphs.
	It follows that every irreducible critically $k$-frustrated signed graph $(G, \sigma)$ satisfies that $d(v) \geq 3$ for each $v \in V(G)$.
	
	\medskip
	Let $\mathcal{L}(k)$ be the family of irreducible critically $k$-frustrated signed graphs and let $\mathcal{L}^*(k)$ be the family of irreducible non-decomposable critically $k$-frustrated signed graphs. 
	
	\begin{theorem}{\rm \cite{CS22}} \label{thm:2family}
		We have $\mathcal{L}(1)=\mathcal{L}^*(1)=\{C_{-1}\}$, $\mathcal{L}(2) = \{ C_{-1} \cup C_{-1} , 2C_{-1}, (K_4, -)\}$ and $\mathcal{L}^*(2)=\{(K_4, -)\}$.
	\end{theorem}
	
	Here $C_{-1}$ is the signed graph on one vertex with a negative loop, $C_{-1} \cup C_{-1}$ is two disjoint copies of it, and $2C_{-1}$ is the signed graph on one vertex with two negative loops on it.
	
	We define a \emph{$(K_4, -)$-subdivision} to be a signed subdivision of $K_4$ with a signature such that each cycle corresponding to a triangle of $K_4$ is negative. 
	Note that if $(H, \pi)$ is a $(K_4, -)$-subdivision, then $\ell(H, \pi)=2$. The following is one of the first structural results on $k$-frustrated signed graphs.
	
	\begin{theorem} \label{thm:Geelen-Guenin-Schrijver1}{\rm \cite{GG01}}
		If a $k$-frustrated signed graph contains no $(K_4, -)$-subdivision, then it contains $k$ edge-disjoint negative cycles.
	\end{theorem} 
	
	We have seen that each of the families $\mathcal{L}^*(1)$ and $\mathcal{L}^*(2)$ is quite small and is precisely described. For $k \geq 3$, we conjecture the following. 

	\begin{conjecture}\label{conj:Lk}
		The set $\mathcal{L}^*(k)$ is finite for any positive integer $k$. 
	\end{conjecture}

	A special subclass $\mathcal{S}^*(k)$ of $\mathcal{L}^*(k)$ consists of those elements $(G, \sigma)$ in $\mathcal{L}^*(k)$ satisfying that for each integer $m\leq k$, every critically $m$-frustrated subgraph of $(G, \sigma)$ is non-decomposable. A relaxation of Conjecture~\ref{conj:Lk} is that the set $\mathcal{S}^*(k)$ is finite.

	\begin{conjecture}\label{conj:Sk}
		The set $\mathcal{S}^*(k)$ is finite for any positive integer $k$. 
	\end{conjecture}
	
	A restriction of Conjecture~\ref{conj:Lk} to signed plane graphs is also of special interest. More precisely, we ask:
	
	\begin{conjecture}\label{conj:Pk}
		Every non-decomposable critically $k$-frustrated signed plane graph has exactly $2k$ facial cycles each of which is a negative cycle. 
	\end{conjecture}
	
	A similar conjecture on the structure of critically $k$-frustrated signed graphs is the following.  
	
	\begin{conjecture}\label{conj:degree2k}
		Every critically $k$-frustrated signed graph $(G, \sigma)$ satisfies $\Delta(G) \leq 2k$. Moreover, $\Delta(G)= 2k$ can only happen if $G$ consists of $k$ negative cycles pairwise edge-disjoint but all containing the same vertex.
	\end{conjecture}

	In this work, we verify Conjectures~\ref{conj:Sk} and \ref{conj:Pk} for $k=3$. The case $k=3$ of Conjecture~\ref{conj:Lk} will be addressed in a forthcoming paper.
    We show that the condition of being non-decomposable is necessary in Conjectures~\ref{conj:Pk} and \ref{conj:Lk}. To support Conjecture~\ref{conj:degree2k}, we prove it is true for a special class of signed graphs.

    \medskip
	The rest of the  paper is organized  as follows. In Section~\ref{sec:S*} we show that the class $\mathcal{S}^*(3)$ contains exactly two elements.
	In Section~\ref{sec:planar} we show that, up to decomposition and subdivision, there exist ten critically $3$-frustrated signed planar graphs. In Section~\ref{sec:Examples} we present a construction of non-decomposable critically frustrated signed graphs from two given such signed graphs. Moreover, we provide an infinite family of decomposable critically $3$-frustrated signed graphs showing that in both Conjectures~\ref{conj:Lk} and \ref{conj:Pk} the assumption of being non-decomposable is necessary even for $k=3$. In Section~\ref{sec:degree}, we give a family of critically $k$-frustrated signed graphs having maximum degree at most $2k$.

 \section[S3]{Characterization of the family $\mathcal{S}^*(3)$}\label{sec:S*}

	To characterize the elements of $\mathcal{S}^*(3)$, we use the following results.
	
	\begin{proposition}{\em \cite{CS22}}\label{charact_S*}
		Let $(G,\sigma)$ be an irreducible critically $k$-frustrated signed graph for some positive integer $k$. Then $(G,\sigma) \in \mathcal{S}^*(k)$ if and only if $(G,\sigma)$ contains no pair of edge-disjoint negative cycles.
	\end{proposition}
	
	Using this fact and the characterization of signed graphs in which every pair of negative cycles intersect, provided in \cite{Lu2018}, the elements of $\mathcal{S}^*(k)$ can be characterized as follows.
	
	\begin{theorem}{\em \cite{CS22}}\label{thm:CS*2}
		Let $(G, \sigma)$ be a signed graph  in $\mathcal{S}^*(k)$ where $\sigma$ is a minimum signature with $E^-_\sigma = \{ x_1y_1, \dots, x_ky_k\}$. Then we have the following.
		\begin{enumerate}[label=(\arabic*)]
			\item\label{Con1} Either $k=1$, in which case $(G,\sigma)$ is a negative loop. 
			\item\label{Con2} Or $k\geq 2$, $G$ is a cubic projective plane graph where the edges passing through the cross cap are those of $E^-_\sigma$.	
		\end{enumerate}		
	\end{theorem}
	
	Let $(G, \sigma) \in \mathcal{S}^*(k)$ with $\sigma$ being a minimum signature. An embedding of $(G, \sigma)$ into the projective plane as described in Theorem~\ref{thm:CS*2} is called a \emph{canonical projective-planar embedding} of $(G,\sigma)$.
 
    Let $(G, \sigma) \in \mathcal{S}^*(k)$ be a canonically projective-planar embedded signed graph. When the choice of the minimum signature $\sigma$ is clear from the context, we will denote the subgraph $G-E^-_\sigma$ of $G$ by $G'$. Moreover, $G'$ will always be considered together with its planar embedding that is implied from Theorem~\ref{thm:CS*2}. The facial cycle of the outer face of this plane graph $G'$ will be denoted by $C_O$. One may observe that in $G'$ the vertices $x_1, \ldots, x_k, y_1, \ldots, y_k$ are all of degree $2$ and they appear on $C_O$ in this cyclic order.  
	
	Given vertices $u$ and $v$ of $C_O$, by $A_{uv}$ we denote the path on $C_O$ connecting $u$ to $v$ which is in the clockwise direction starting at $u$ and ending at $v$. When referring to a face of $G'$ we do not consider the outer face. Thus a face $F$ of $G'$ is also a face of $(G, \sigma)$ in the projective-planar embedding from which $G'$ is defined. The boundary of this face $F$, which must be a cycle, will be denoted by $C_F$. 
	
	A face of $G'$ is said to be \emph{internal} if its boundary shares no edge with $C_O$. We note that, since $G'$ is subcubic, the boundary of an internal face does not intersect $C_O$ at a vertex either. 
	In particular, the boundary of a face $F$ which is not internal shares at least two vertices with $C_O$. We classify such faces depending on how many of those common vertices are in the set $\mathcal{R}=\{x_1,...,x_k, y_1,...,y_k\}$.  More precisely, a face $F$ is said to be an \emph{$i$-face} of $G'$ if $C_F$ contains $i$ elements from the set $\mathcal{R}$. Two faces $F_1$ and $F_2$ are said to be \emph{adjacent on the boundary} if $V(C_{F_1} \cap C_{F_2}\cap C_O) \neq \emptyset$.  A face $F$ of $G'$ is called a \emph{bridge-face} if the subgraph induced by $C_F \cap C_O$ is disconnected. See Figure~\ref{fig:bridge} for an example, noting that curves represent paths that might contain more vertices.

	\begin{figure}[htbp]
		\centering
		\begin{minipage}[t]{.4\textwidth}
			\centering	
           \begin{tikzpicture}[scale=.74]
               \draw[fill=none, line width=0.4mm, blue](0,0) circle (2.5)  {};
				\draw[rotate=15] (0,2.5) node[circle, draw=black!80, inner sep=0mm, minimum size=3.5mm, fill=white] (a_{1}){};
				\draw[rotate=-15] (0,2.5) node[circle, draw=black!80, inner sep=0mm, minimum size=3.5mm, fill=white] (a_{2}){};
				\draw[rotate=-75] (0,2.5) node[circle, draw=black!80, inner sep=0mm, minimum size=3.5mm, fill=white] (a_{3}){};
				\draw[rotate=-105] (0,2.5) node[circle, draw=black!80, inner sep=0mm, minimum size=3.5mm, fill=white] (a_{4}){};
				\draw[rotate=-150] (0,2.5) node[circle, draw=black!80, inner sep=0mm, minimum size=3.5mm, fill=white] (a_{5}){};
				\draw[rotate=180] (0,2.5) node[circle, draw=black!80, inner sep=0mm, minimum size=3.5mm, fill=white] (a_{6}){};
				\draw[rotate=90] (0,2.5) node[circle, draw=black!80, inner sep=0mm, minimum size=3.5mm, fill=white] (a_{7}){};
				\draw[rotate=50] (0,2.5) node[circle, draw=black!80, inner sep=0mm, minimum size=3.5mm, fill=white] (a_{8}){};
				\foreach \i/\j in {2/3,4/5,6/7,8/1} {
					\draw[bend right=60, line width=0.4mm, blue] (a_{\i}) to  (a_{\j});
				}	
				\node[text=blue] at (.1,.2) {$F$};

			\end{tikzpicture}
			\caption{A bridge-face $F$ in $G'$}
			\label{fig:bridge}
		\end{minipage}
		\begin{minipage}[t]{.4\textwidth}
			\centering
			\begin{tikzpicture}[scale=.74]
				\draw[fill=none, line width=0.4mm, blue](0,0) circle (2.5) {};
				\draw[rotate=30] (0,2.5) node[circle, draw=black!80, inner sep=0mm, minimum size=3.5mm, fill=white] (a_{1}){\scriptsize $a_1$};
				\draw[rotate=-10] (0,2.5) node[circle, draw=black!80, inner sep=0mm, minimum size=3.5mm, fill=white] (a_{2}){\scriptsize $a_2$};
				\draw[rotate=-80] (0,2.5) node[circle, draw=black!80, inner sep=0mm, minimum size=3.5mm, fill=white] (a_{3}){\scriptsize $b_1$};
				\draw[rotate=-110] (0,2.5) node[circle, draw=black!80, inner sep=0mm, minimum size=3.5mm, fill=white] (a_{4}){\scriptsize $b_2$};
				\draw[rotate=-150] (0,2.5) node[circle, draw=black!80, inner sep=0mm, minimum size=3.5mm, fill=white] (a_{5}){};
				\draw[rotate=180] (0,2.5) node[circle, draw=black!80, inner sep=0mm, minimum size=3.5mm, fill=white] (a_{6}){};
				\draw[rotate=90] (0,2.5) node[circle, draw=black!80, inner sep=0mm, minimum size=3.5mm, fill=white] (a_{7}){};
				\draw[rotate=55] (0,2.5) node[circle, draw=black!80, inner sep=0mm, minimum size=3.5mm, fill=white] (a_{8}){};
				
				\foreach \i/\j in {2/3,4/5,8/1} {
					\draw[bend right=60, line width=0.4mm, blue] (a_{\i}) to  (a_{\j});
				}

                \draw[rotate=10] (0,2.5) node[circle, draw=black!80, inner sep=0mm, minimum size=3.5mm, fill=white] (x_{1}){\scriptsize $x_1$};
				\draw[rotate=-30] (0,2.5) node[circle, draw=black!80, inner sep=0mm, minimum size=3.5mm, fill=white] (x_{2}){\scriptsize $x_2$};	
				\draw[rotate=-55] (0,2.5) node[circle, draw=black!80, inner sep=0mm, minimum size=3.5mm, fill=white] (x_{3}){\scriptsize $x_3$};
				
				\draw[bend right=60, line width=0.4mm, blue, dotted] (a_{6}) to  (a_{7});

				\node[text=blue] at (.1,.2) {$F$};

			\end{tikzpicture}
   \caption{Proposition~\ref{prop:i-face}}
			\label{fig:Claim_i-face}
		\end{minipage}
	\end{figure}

	Note that in $(G, \sigma)$, each edge-cut with negative edges contains at least two edges of $C_O$. Furthermore, based on the cyclic order of the elements of $\mathcal{R}$ on $C_O$, we have the following observation.
	\begin{observation}\label{obs:cutG'}
		Let $(G, \sigma)$ be a canonically projective-planar embedded signed graph in $\mathcal{S}^*(k)$ for $k\geq 2$. If an edge-cut $\partial_{G}(X)$ contains exactly two edges $e_1$ and $e_2$ of $C_O$, then $d^-_{(G, \sigma)}(X)= \min \{|V(A_1) \cap \mathcal{R}|, |V(A_2) \cap \mathcal{R}| \}$ where $A_1$ and $A_2$ are the two connected components of $C_O-\{e_1, e_2\}$.
	\end{observation}
	
	\begin{lemma}\label{lem:2or2k-2}
		Let $(G, \sigma)$ be a canonically projective-planar embedded signed graph in $\mathcal{S}^*(k)$ for $k\geq 2$.
		Assume that $F$ is a bridge-face and let $A_{a_1a_2}$ and $A_{b_1 b_2}$ be two connected components of $C_F \cap C_O$ such that $A_{a_2 b_1}$ is a connected component in $C_O\setminus C_F$.
		Then $|V(A_{a_2 b_1}) \cap \mathcal{R}| \in \{2,2k-2\}$.
	\end{lemma}
	
	\begin{proof}
		Let $e_1$ (resp. $e_2$) be the edge in $A_{a_1a_2}$ (resp. $A_{b_1 b_2}$) that has $a_2$ (resp. $b_1$) as an endpoint. Let $G''$ be the connected component of $G'-\{e_1,e_2\}$ containing $a_2$ (and $b_1$). 
		
		We first show that $|V(A_{a_2 b_1}) \cap \mathcal{R}| \not\in \{3, 4, \ldots, 2k-3\}$. Otherwise, by Observation~\ref{obs:cutG'} the edge-cut $\partial_{G}(V(G''))$ must contain at least three negative edges, but it has only two positive edges. This contradicts the fact that $\sigma$ is a minimum signature.
		
		Next we show that $|V(A_{a_2 b_1}) \cap \mathcal{R}| \leq 1$ is not possible either. That $|V(A_{a_2 b_1}) \cap \mathcal{R}| \geq 2k-1$ is not possible follows similarly. Suppose to the contrary that $|V(A_{a_2 b_1}) \cap \mathcal{R}| \leq 1$. In $C_F-\{e_1,e_2\}$, there is a path connecting $a_2$ to $b_1$ and let $e$ be an edge of this path. By criticality, there exists an equilibrated edge-cut $\partial (X)$ containing $e$.
		Since each equilibrated cut of $(G, \sigma)$ contains at least two (positive) edges from $C_O$ and noting that $e$ is also a positive edge, we have $d^+(X)\geq 3$, and hence $d^-(X)\geq 3$. Moreover, by the choice of $e$, at least one of the edges of $A_{a_2b_1}$, say $e'$, is in $\partial (X)$. Thus in total, at least two edges of $G''$ are in $\partial (X)$. We now consider the following two edge-cuts: $E_1=\partial (X)\setminus E(G'') \cup \{e_1\}$ and  $E_2=\partial (X)\setminus E(G'') \cup \{e_2\}$. Since $|V(A_{a_2 b_1}) \cap \mathcal{R}| \leq 1$, it follows that one of these two edge-cuts say $E_1$, has the same set of negative edges as $\partial (X)$. However, $E_1$ has fewer positive edges than $\partial (X)$, contradicting the minimality of $\sigma$.
  \end{proof}

	\begin{proposition}\label{prop:i-face}
		Let $(G, \sigma)$ be a canonically projective-planar embedded signed graph in $\mathcal{S}^*(k)$ for $k\geq 3$. Then we have the following:
		\begin{enumerate}[label=(\roman*)]
			\item\label{F1} Every bridge-face of $G'$ is a $0$-face.
			\item\label{F2} For $i\geq 3$ there is no $i$-face in $G'$. 
		\end{enumerate}
	\end{proposition}
	
	\begin{proof}
		\ref{F1} Let $F$ be a bridge-face of $G'$ and assume that $C_F \cap C_O$ consists of $t$ connected components (thus $t\geq 2$). Let $A_{a_1a_2}$ and $A_{b_1 b_2}$ be two connected components of $C_F \cap C_O$ such that $A_{a_2 b_1}$ is a connected component in $C_O\setminus C_F$. By Lemma~\ref{lem:2or2k-2}, $|V(A_{a_2 b_1}) \cap \mathcal{R}| \in \{2,2k-2\}$. Toward a contradiction and without loss of generality, assume that $x_1 \in \mathcal{R}\cap  V(A_{a_1a_2})$ and $x_2,x_3\in \mathcal{R}\cap  V(A_{a_2b_1})$, depicted in Figure~\ref{fig:Claim_i-face}.

        We claim that each connected component of $C_O\setminus C_F$ contains exactly two vertices from $\mathcal{R}$. If not, then one of them contains $2k-2$ vertices from $\mathcal{R}$. In this case, since $|\mathcal{R}|=2k$, there is only one other component in $C_O\setminus C_F$. Furthermore, this component must contain the other two vertices of $\mathcal{R}$. This in turn implies that $F$ is a $0$-face.  
		
		Let $e_1$ be the edge on $A_{a_1x_1}$ incident with $x_1$ and let $e_2$ be the edge on $A_{b_1b_2}$ incident with $b_1$. Then the set $\{e_1, e_2, x_1y_1, x_2y_2, x_3y_3\}$ is an edge-cut consisting of two positive edges and three negative edges, contradicting the fact that $\sigma$ is a minimum signature. 
		
		\smallskip
		\noindent
		\ref{F2} 
		Suppose to the contrary that $F$ is an $i$-face of $G'$ for $i \geq 3$. 
		By Claim~\ref{F1}, we know that $F$ is not a bridge-face. Therefore, by the symmetry of labeling, we assume that $x_1, x_2, x_3\in V(C_F)\cap \mathcal{R}$.
		Let $e_1=vx_1, e_2=x_3u \in E(C_F\cap C_O)$ such that $v\not\in V(A_{x_1x_2})$ and $u\not\in V(A_{x_2x_3})$. Then the edge set $ \{e_1, e_2, x_1y_1, x_2y_2, x_3y_3\}$ is an edge-cut that contains three negative edges but only two positive edges, a contradiction. 
	\end{proof}
	
	From now on, we focus on the family $\mathcal{S}^*(3)$. We give some structural properties of signed graphs in $\mathcal{S}^*(3)$ in the following lemmas.

	\begin{lemma}\label{lem:1-2face}
		Let $(G, \sigma)$ be a canonically projective-planar embedded signed graph in $\mathcal{S}^*(3)$. Then each face of $G'$ is either a bridge-face or an $i$-face for $i \in \{1,2\}$.
	\end{lemma}
	
	\begin{proof}
		By Proposition~\ref{prop:i-face}~\ref{F2}, if $F$ is an $i$-face of $G'$, then $i\in \{0,1,2\}$. It remains to show that there are no internal faces and that every $0$-face is a bridge-face.
		
		For the first claim, assume to the contrary that there exists an internal face $F$ of $G'$. Note that each equilibrated cut containing one edge of $C_F$ must have at least two (positive) edges from $C_F$ and two (positive) edges from $C_O$. However, there are only three negative edges in $(G, \sigma)$, contradicting the fact that each equilibrated cut has the same number of positive and negative edges.
		
		For the second claim, assume that $F$ is a $0$-face of $G'$ which is not a bridge-face. As $C_F$ shares at least one edge with the outer facial cycle $C_O$ of $G'$, there is a face $F'$ such that $C_{F'}$ shares a common vertex with both $C_F$ and $C_O$. Assume that $F'$ is an $i$-face for $i\in \{0,1,2\}$. Let $e_0$ be a (positive) edge in the path $C_F \cap C_{F'}$. Let $\partial(X)$ be the equilibrated cut containing $e_0$. Recall that any equilibrated cut must contain at least two edges of $C_O$. As $(G, \sigma) \in \mathcal{S}^*(3)$, $\partial(X)$ contains exactly two edges of $C_O$. Furthermore, one of these two edges belongs to $E(C_O\cap C_F)$ while the other is in $E(C_O\cap C_{F'})$. To complete the proof, it suffices to show that $|X\cap \mathcal{R}|\neq 3$, which would contradict the fact that $\partial(X)$ is an equilibrated cut. If $F'$ is not a bridge-face, then by Proposition~\ref{prop:i-face} there are at most two elements of $\mathcal{R}$ in $C_{F'}$, and consequently at most two elements of $\mathcal{R}$ in $X$ (i.e., $|X\cap \mathcal{R}|\leq 2$). If $F'$ is a bridge-face, then by Lemma~\ref{lem:2or2k-2} the number of elements of $\mathcal{R}$ in each connected component of $C_O\setminus C_{F'}$ is either $2$ or $4$. As either all of the vertices of a connected component of $C_O\setminus C_{F'}$ are contained in $X$ or none of them is in $X$, $|X \cap \mathcal{R}|$ has to be an even number and clearly $|X\cap \mathcal{R}|\neq 3$.
	\end{proof}

	\begin{figure}[htbp]
		\centering
		\begin{minipage}[t]{.4\textwidth}
			\centering
			\begin{tikzpicture}[scale=.38]\
                \draw[fill=none, line width=0.4mm, blue](0,0) circle (5)  {};
				\draw[rotate=360-36*(0)] (0,5) node[circle, draw=black!80, inner sep=0mm, minimum size=3.5mm, fill=white] (x_{0}){\scriptsize $a_2$};
				\foreach \i in {1,2}
				{
					\draw[rotate=360-36*(\i)] (0,5) node[circle, draw=black!80, inner sep=0mm, minimum size=3.5mm, fill=white] (x_{\i}){\scriptsize $x_{\i}$};
				}
				\draw[rotate=360-36*(3)] (0,5) node[circle, draw=black!80, inner sep=0mm, minimum size=3.5mm, fill=white] (x_{3}){\scriptsize $b_1$};
				\draw[rotate=360-36*(4)] (0,5) node[circle, draw=black!80, inner sep=0mm, minimum size=3.5mm, fill=white] (x_{4}){\scriptsize $b_2$};
				\draw[rotate=360-36*(5)] (0,5) node[circle, draw=black!80, inner sep=0mm, minimum size=3.5mm, fill=white] (x_{5}){\scriptsize $x_3$};
				\draw[rotate=360-36*(6)] (0,5) node[circle, draw=black!80, inner sep=0mm, minimum size=3.5mm, fill=white] (x_{6}){\scriptsize $y_1$};
				\draw[rotate=360-36*(7)] (0,5) node[circle, draw=black!80, inner sep=0mm, minimum size=3.5mm, fill=white] (x_{7}){\scriptsize $y_2$}; 
				\draw[rotate=360-36*(8)] (0,5) node[circle, draw=black!80, inner sep=0mm, minimum size=3.5mm, fill=white] (x_{8}){\scriptsize $y_3$};
				\draw[rotate=360-36*(9)] (0,5) node[circle, draw=black!80, inner sep=0mm, minimum size=3.5mm, fill=white] (x_{9}){\scriptsize $a_1$};
				
				\draw[ line width=0.4mm, blue] (x_{9}) edge[bend right=30] (x_{4});	
				
				\draw[ line width=0.4mm, blue] (x_{0}) edge[bend right=30] node [midway, draw, fill=black, rectangle, label=right:$e_0$]  {} (x_{3});	
				
				\foreach \i in {2,7}
				{						
					\draw[rotate=360-36*(\i)+6] (0,7.8) node[circle, draw=black!80, inner sep=0mm, minimum size=0mm] (z_{\i}){};
				}
				
				\foreach \i in {3,8}
				{						
					\draw[rotate=360-36*(\i)-6] (0,7.8) node[circle, draw=black!80, inner sep=0mm, minimum size=0mm] (z_{\i}){};
				}
				
				\foreach \i in {0,5}
				{
					\draw[rotate=360-36*(\i)] (0,7.8) node[circle, draw=black!80, inner sep=0mm, minimum size=0mm] (z_{\i}){};
				}
				\node[text=blue] at (-.7,.8) {$F$};
			\end{tikzpicture}
			\caption{Case in Lemma~\ref{lemma:bridge-face}}
			\label{fig:Bridge3Comp}
		\end{minipage}
		\begin{minipage}[t]{.4\textwidth}
			\centering
			\begin{tikzpicture}[scale=.38]
                \draw[fill=none, line width=0.4mm, blue](0,0) circle (5)  {};
				\draw[rotate=360-40*(0)] (0,5) node[circle, draw=black!80, inner sep=0mm, minimum size=3.5mm, fill=white] (x_{0}){\scriptsize $x_1$};
				\draw[rotate=360-40*(1)] (0,5) node[circle, draw=black!80, inner sep=0mm, minimum size=3.5mm, fill=white] (x_{1}){};
				\draw[rotate=360-40*(2)] (0,5) node[circle, draw=black!80, inner sep=0mm, minimum size=3.5mm, fill=white] (x_{2}){\scriptsize $x_2$};      
				\draw[rotate=360-40*(3)] (0,5) node[circle, draw=black!80, inner sep=0mm, minimum size=3.5mm, fill=white] (x_{3}){};
				\draw[rotate=360-40*(4)] (0,5) node[circle, draw=black!80, inner sep=0mm, minimum size=3.5mm, fill=white] (x_{4}){\scriptsize $x_3$};
				\draw[rotate=360-40*(5)] (0,5) node[circle, draw=black!80, inner sep=0mm, minimum size=3.5mm, fill=white] (x_{5}){\scriptsize $y_1$};
				\draw[rotate=360-40*(6)] (0,5) node[circle, draw=black!80, inner sep=0mm, minimum size=3.5mm, fill=white] (x_{6}){\scriptsize $y_2$};
				\draw[rotate=360-40*(7)] (0,5) node[circle, draw=black!80, inner sep=0mm, minimum size=3.5mm, fill=white] (x_{7}){\scriptsize $y_3$}; 
				\draw[rotate=360-40*(8)] (0,5) node[circle, draw=black!80, inner sep=0mm, minimum size=3.5mm, fill=white] (x_{8}){};
				
				
				\draw[ line width=0.4mm, blue] (x_{8}) edge[bend right=60] node [near end, draw, fill=black, rectangle, label=right:$e_0$]  {} (x_{1});
				
				\draw[ line width=0.4mm, blue] (-0.3,1.9) edge[bend right=60] node [near end, above=10pt]  {$F_2$} (x_{3});
				\node[text=blue] at (0,3.2) {$F_1$};	
				\foreach \i in {1,4,7}
				{
					\draw[rotate=360-40*(\i)+12] (0,7.8) node[circle, draw=black!80, inner sep=0mm, minimum size=0mm] (z_{\i}){};
				}
				
				\foreach \i in {2,5,8}
				{
					\draw[rotate=360-40*(\i)-12] (0,7.8) node[circle, draw=black!80, inner sep=0mm, minimum size=0mm] (z_{\i}){};
				}
				
				\draw[rotate=360-40*(0)] (0,-7.1) node[circle, draw=white, inner sep=0mm, minimum size=3.9mm] (){};
				
			\end{tikzpicture}
			\caption{Case in Lemma~\ref{lem:1-face2-face}}
			\label{fig:i+j_2}
		\end{minipage}
	\end{figure}

	\begin{lemma}\label{lemma:bridge-face}
		Let $(G, \sigma)$ be a canonically projective-planar embedded signed graph in $\mathcal{S}^*(3)$. If $F$ is a bridge-face of $G'$, then $C_F \cap C_O$ has exactly three connected components. In particular, there is at most one bridge-face.
	\end{lemma}
	
	\begin{proof}
		As $(G, \sigma) \in \mathcal{S}^*(3)$, we have $|\mathcal{R}|=6$. As $F$ is a bridge-face, $C_F \cap C_O$ has at least two components, and, by Lemma~\ref{lem:2or2k-2}, has at most three components. It remains to show that $C_F \cap C_O$ does not have exactly two components. Assume to the contrary that $C_F \cap C_O$ has exactly two components, say $A_{a_1a_2}$ and $A_{b_1b_2}$. Then one of $A_{a_2b_1}$ or $A_{b_2a_1}$, say $A_{a_2b_1}$ without loss of generality, has two elements from $\mathcal{R}$, and the other, $A_{b_2a_1}$ in this case, has four elements from $\mathcal{R}$. See Figure~\ref{fig:Bridge3Comp} for a depiction.
		
		Let $e_0$ be an edge on the $a_2b_1$-path of $C_F$ which is internally vertex-disjoint from $C_O$. Let $\partial (X)$ be an equilibrated cut containing $e_0$. As $\partial(X)$ must contain two (positive) edges, say $e_1$ and $e_2$, of $C_O$, it has to be an edge-cut of size $6$ and hence $e_0$, $e_1$, and $e_2$ are the only positive edges of it. Thus one of $e_1$ or $e_2$ is on $A_{a_2b_1}$ and the other is on $A_{a_1a_2}\cup A_{b_1b_2}$. Noting that each bridge-face is a $0$-face by Proposition~\ref{prop:i-face}~\ref{F1} and $A_{a_2b_1}$ contains two elements from $\mathcal{R}$, $X$ has at most two vertices of $\mathcal{R}$ and, therefore, $\partial(X)$ contains at most two negative edges, contradicting the fact that it is an equilibrated cut.
		
		Finally, by Lemma~\ref{lem:2or2k-2}, as each of the connected components of $C_O\setminus C_F$ must contain either two or four elements of $\mathcal{R}$, and since there are three connected components, each of them contains exactly two elements of $\mathcal{R}$ and thus there is no other bridge-face.
	\end{proof}

	\begin{lemma}\label{lem:1-face2-face}
		Let $(G, \sigma)$ be a canonically projective-planar embedded signed graph in $\mathcal{S}^*(3)$. Let $F_1$ and $F_2$ be an $i_1$-face and an $i_2$-face of $G'$, respectively. If $F_1$ is adjacent to $F_2$ on the boundary, then either (i) $i_1+i_2 \geq 3$ or (ii) one of $F_1$ and $F_2$ is a bridge-face.
	\end{lemma}
	
	\begin{proof}
		Assume that neither of $F_1$ and $F_2$ is a bridge-face. 
		By Lemma~\ref{lem:1-2face} $i_1+i_2 \geq 2$, and it remains to prove that $i_1+i_2 \neq 2$. Assume to the contrary that $i_1+i_2 =2$. Let $e_0$ be an edge on the path $C_{F_1} \cap C_{F_2}$. See Figure~\ref{fig:i+j_2}. Each equilibrated cut containing the edge $e_0$ must have two more (positive) edges of $C_O$ say $e_1$ and $e_2$. It follows as before that $e_1$ is on $C_{F_1}\cap C_O$ and $e_2$ is on $C_{F_2}\cap C_O$. Since $i_1+i_2=2$, a similar argument implies that $X$ can contain at most two vertices from $\mathcal{R}$, leading to a contradiction with $\partial(X)$ being an equilibrated cut.
	\end{proof}

	We are now ready to give the full description of $\mathcal{S}^*(3)$.
	
	\begin{figure}[htbp]
		\centering
		\begin{subfigure}[t]{.65\textwidth}
			\centering
			\begin{minipage}[t]{.5\textwidth}
				\centering
				\begin{tikzpicture}
					[scale=.3]
					\foreach \i in {0,1,2,3,4,5,6,7,8,9,10,11}
					{
						\draw[rotate=360-30*(\i)+15] (0,5) node[circle, draw=black!80, inner sep=0mm, minimum size=3.3mm] (x_{\i}){\scriptsize ${\i}$};
					}
					\foreach \i/\j in {0/1,1/2,2/3,3/4,4/5,5/6,6/7,7/8,8/9,9/10,10/11,11/0}
					{
						\draw[line width=0.4mm, blue] (x_{\i}) -- (x_{\j});
					} 
					\foreach \i/\j in {1/4,5/8,9/0}
					{
						\draw[line width=0.4mm, blue] (x_{\i}) -- (x_{\j});
					}
					
					\foreach \i in {2,6,10}
					{
						\draw[rotate=360-30*(\i)+30] (0,8) node[circle, draw=black!80, inner sep=0mm, minimum size=0mm] (z_{\i}){};
					}
					
					\foreach \i in {3,7,11}
					{
						\draw[rotate=360-30*(\i)] (0,8) node[circle, draw=black!80, inner sep=0mm, minimum size=0mm] (z_{\i}){};
					}
					
					\foreach \i in {2,7,3,10,6,11}
					{
						\draw[dash pattern=on 1mm off .7mm, line width=0.4mm, red] (x_{\i}) -- (z_{\i});
					}

					\node[circle,draw, dotted, line width=0.3mm, inner sep=16.5mm] (CC) at (0,0){};

				\end{tikzpicture}
			\end{minipage}
			\begin{minipage}[t]{.4\textwidth}
				\centering
				\begin{tikzpicture}
					[scale=.3]
					\foreach \i in {0,1,2,3,4,5,6,7,8,9}
					{
						\draw[rotate=360-36*(\i)] (0,5) node[circle, draw=black!80, inner sep=0mm, minimum size=3.3mm] (x_{\i}){\scriptsize ${\i}$};
					}
					\foreach \i in {1,2}
					{
						\draw[rotate=360-180*(\i)] (0,2) node[circle, draw=black!80, inner sep=0mm, minimum size=3.3mm] (y_{\i}){\scriptsize $w_{\i}$};
					}
					\foreach \i/\j in {0/1,1/2,2/3,3/4,4/5,5/6,6/7,7/8,8/9,9/0}
					{
						\draw[line width=0.4mm, blue] (x_{\i}) -- (x_{\j});
					} 
					
					\foreach \i/\j in {1/2,9/2,4/1,6/1}
					{
						\draw[line width=0.4mm, blue] (x_{\i}) -- (y_{\j});
					} 
					\draw[line width=0.4mm, blue] (y_{1}) -- (y_{2});

					\foreach \i in {2,7}
					{						
						\draw[rotate=360-36*(\i)+6] (0,7.8) node[circle, draw=black!80, inner sep=0mm, minimum size=0mm] (z_{\i}){};
					}
					
					\foreach \i in {3,8}
					{						
						\draw[rotate=360-36*(\i)-6] (0,7.8) node[circle, draw=black!80, inner sep=0mm, minimum size=0mm] (z_{\i}){};
					}
					
					\foreach \i in {0,5}
					{
						\draw[rotate=360-36*(\i)] (0,7.8) node[circle, draw=black!80, inner sep=0mm, minimum size=0mm] (z_{\i}){};
					}
					
					\foreach \i in {0,5, 2,3, 7,8}
					{
						\draw[dash pattern=on 1mm off .7mm,, line width=0.4mm, red] (x_{\i}) -- (z_{\i});
					}
					
					\node[circle,draw, dotted, line width=0.3mm, inner sep=16.5mm] (CC) at (0,0){};

				\end{tikzpicture}
			\end{minipage}
			\caption{$\hat{G}_1$ with two embeddings}
			\label{fig:G1}
		\end{subfigure}
		\begin{subfigure}[t]{.33\textwidth}
			\centering
			\begin{tikzpicture}[scale=.3]
				\draw(0,0) node[circle, draw=black!80, inner sep=0mm, minimum size=3.3mm] (w){\scriptsize $w$};
				\foreach \i in {1,2,3,4,5,6,7,8,9}
				{
					\draw[rotate=360-40*(\i)] (0,5) node[circle, draw=black!80, inner sep=0mm, minimum size=3.3mm] (x_{\i}){\scriptsize ${\i}$};
				}
				
				\foreach \i/\j in {1/2,2/3,3/4,4/5,5/6,6/7,7/8,8/9,9/1}
				{
					\draw[line width=0.4mm, blue] (x_{\i}) -- (x_{\j});
				} 
				\foreach \i in {3,6,9}
				{
					\draw  [line width=0.4mm, blue] (x_{\i}) -- (w);
				}

				\foreach \i in {1,4,7}
				{
					\draw[rotate=360-40*(\i)+12] (0,7.8) node[circle, draw=black!80, inner sep=0mm, minimum size=0mm] (z_{\i}){};
				}
				
				\foreach \i in {2,5,8}
				{
					\draw[rotate=360-40*(\i)-12] (0,7.8) node[circle, draw=black!80, inner sep=0mm, minimum size=0mm] (z_{\i}){};
				}
				
				\foreach \i in {1,2,4,5, 7,8}
				{
					\draw[dash pattern=on 1mm off .7mm, line width=0.4mm, red] (x_{\i}) -- (z_{\i});
				}
				
				\node[circle,draw, dotted, line width=0.3mm, inner sep=16.5mm] (CC) at (0,0){};
				
			\end{tikzpicture}
			\caption{$\hat{G}_2$}
			\label{fig:Petersen}
		\end{subfigure}
		\caption{$\mathcal{S}^*(3)$}
		\label{fig:S*3}
	\end{figure}
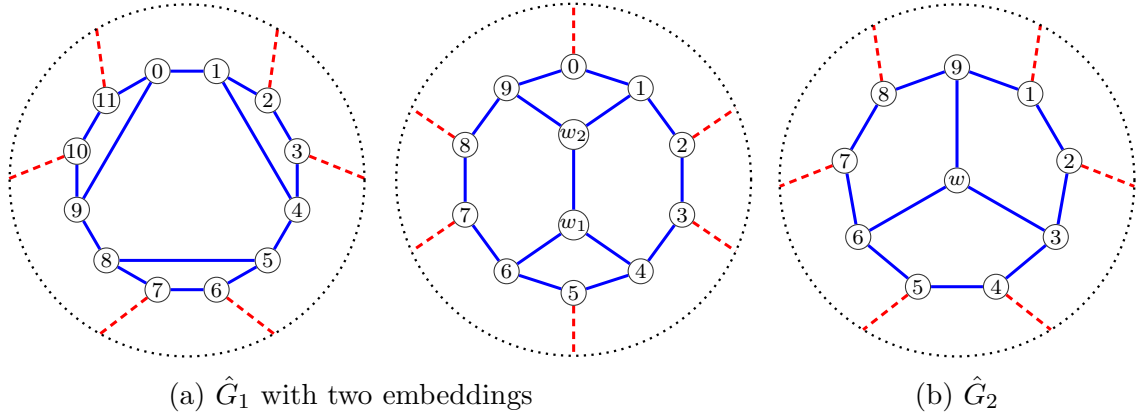

	\begin{theorem}
		The class $\mathcal{S}^*(3)$ consists of two signed graphs, depicted in Figure~\ref{fig:S*3}.
	\end{theorem}
	
	\begin{proof}
		We consider the following three cases:
		\begin{itemize}
			\item $G'$ has a bridge-face $F$. By Lemma~\ref{lemma:bridge-face}, $F$ is the only bridge-face of $G'$ and $C_O\setminus C_{F}$ consists of three components each of which has exactly two elements from $\mathcal{R}$. Furthermore, it follows from Lemma~\ref{lem:1-face2-face} that the vertices of $\mathcal{R}$ are the only vertices on each of these components, as otherwise a vertex not in $\mathcal{R}$ would result in an $i_1$-face and an $i_2$-face with $i_1+i_2\leq 2$. This leads to the projective planar graph of Figure~\ref{fig:G1} (left).

			\item $G'$ has at least one $1$-face (and no bridge-face). Let $F_1$ be a $1$-face of $G'$. As $G'$ has no bridge-face, by Lemma~\ref{lem:1-2face}, it has no $0$-face. Furthermore, by Lemma~\ref{lem:1-face2-face}, each of the two faces adjacent to $F_1$ on the boundary are $2$-faces. As there are only six vertices in $\mathcal{R}$, and as there is no internal face by Lemma~\ref{lem:1-2face}, there is only one remaining face. Furthermore, this face is a $1$-face. Let $G''$ be the graph obtained from $G'$ by suppressing all vertices of $\mathcal{R}$ and note that $G''$ is cubic and planar.
			
			It then follows from Euler's formula that $|V(G'')|- \frac{3}{2}|V(G'')| +5 =2$, i.e., $|V(G'')|= 6$. But there are only two cubic graphs on $6$ vertices: $K_{3,3}$ and the $3$-prism. As $K_{3,3}$ is not planar, $G''$ is the $3$-prism. As each $1$-face of $G'$ is adjacent to two $2$-faces of $G'$, both of the triangles of $G''$ correspond to faces of the same type in $G'$. More precisely, either each corresponds to a $1$-face or each corresponds to a $2$-face. The former case leads to the projective planar graph of Figure~\ref{fig:G1} (right). In the latter case, we consider the middle edge of the $3$-path and we observe that this edge cannot be in an equilibrated cut. 
			
			\item Each face of $G'$ is a $2$-face. Hence, $G'$ has exactly three $2$-faces. Similar to the previous case we consider the graph $G''$ obtained from $G'$ by suppressing all vertices of $\mathcal{R}$. It follows from Euler's formula that $G''$ has four vertices and noting that $G''$ is cubic, hence, it must be $K_4$. Thus $(G, \sigma)$ is the signed graph in Figure~\ref{fig:Petersen}.
		\end{itemize}
		We note that the two signed graphs in Figure~\ref{fig:G1} are switching-isomorphic and thus up to switching $\mathcal{S}^*(3)$ consists of two signed graphs. 
	\end{proof}
	
	Note that the signed graph $\hat{G}_2$ of Figure~\ref{fig:Petersen} is a signed Petersen graph.

    \section[P3]{Critically $3$-frustrated signed planar graphs}\label{sec:planar}
	
	Let $\mathcal{P}^*(3)$ denote the class of irreducible non-decomposable critically $3$-frustrated signed planar graphs. 
	In this section, we show that each signed plane graph in the class $\mathcal{P}^*(3)$ has exactly six negative facial cycles and no positive facial cycles. Using this we conclude that there are ten non-isomorphic signed graphs (with respect to switching isomorphism) in $\mathcal{P}^*(3)$. They are depicted in Figure~\ref{fig:P*3}.
	
    We will need the next lemma that follows from the description of $\mathcal{L}^{*}(2)$.
	
	\begin{lemma}\label{lem:3NegativeCycles}
		Let $C_1, C_2$, and $C_3$ be three negative cycles of a signed graph $(G, \sigma)$. If $E(C_1)\cap E(C_2)\cap E(C_3)=\emptyset$, then the signed subgraph induced by $C_1, C_2,$ and $C_3$ contains either a $(K_4, -)$-subdivision or two edge-disjoint negative cycles. 
	\end{lemma}
	
	\begin{proof}
		Since $E(C_1)\cap E(C_2)\cap E(C_3)=\emptyset$, the frustration index of the signed subgraph induced by $C_1 \cup C_2 \cup C_3$ is at least $2$. Hence, it contains a critically $2$-frustrated subgraph. The statement then follows from Theorem~\ref{thm:2family}.
	\end{proof}

	Noting that each edge of a plane graph belongs to exactly two facial cycles, we have the following observation, which implies that any element of $\mathcal{P}^*(3)$ has at most six negative facial cycles.
	
	\begin{observation}\label{obs:FacialNegativeCycles}
		Every critically $k$-frustrated signed plane graph has at most $2k$ negative facial cycles. Moreover, if there are $2k$ negative facial cycles, then they are the only facial cycles.
	\end{observation}

	Next, we show that each signed plane graph in $\mathcal{P}^*(3)$ has exactly six negative facial cycles. In fact, we prove this for a larger class of critically $3$-frustrated signed graphs which are not necessarily irreducible.  	
	\begin{theorem}\label{thm:NoPositiveFace}
		Let $(G, \sigma)$ be a non-decomposable critically $3$-frustrated signed plane graph. Then $(G, \sigma)$ consists of six negative facial cycles.
	\end{theorem}

	\begin{proof}
     Since $(G,\sigma)$ is not decomposable, by Theorem~\ref{thm:Geelen-Guenin-Schrijver1} $(G,\sigma)$ contains a $(K_4,-)$-subdivision $(H,\sigma)$ as a subgraph. Let $e_1$ be an edge of $E(G\setminus H)$, noting that it is not an empty set because $\ell(G,\sigma)=3$. Without loss of generality, we assume that $\sigma$ is a minimum signature where $e_1$ is assigned to be negative. We observe that the other two negative edges of $E^-_\sigma$ is on the $(K_4,-)$-subdivision $(H,\sigma)$.
		
	To prove the theorem it suffices to show that each facial cycle of $(G,\sigma)$ contains at most one negative edge. That is because, this together with the fact that $\ell(G,\sigma)=3$ would imply the existence of six negative facial cycles. The claim then follows from Observation~\ref{obs:FacialNegativeCycles}.
        
		As there are only two negative edges in $(H, \sigma)$, say $e_2$ and $e_3$, no facial cycle of $(H, \sigma)$ contains two negative edges. Thus in $(G, \sigma)$ no facial cycle contains three negative edges. It remains to show that no facial cycle of $(G, \sigma)$ contains two negative edges. Assume to the contrary that $C_{F_2}$ is such a facial cycle. As the negative edges cannot be $e_2$ and $e_3$, and by the symmetry between these two labels, we may assume that $e_1$ and $e_2$ are the negative edges of $C_{F_2}$. Let $C_{F_1}$ and $C_{F_3}$ be the other facial cycles incident with $e_1$ and $e_2$, respectively. 
		Observe that $e_3$ neither belongs to $C_{F_1}$ nor to $C_{F_3}$, as otherwise $(G, \sigma)$ has only two negative faces, contradicting the fact that it contains a $(K_4,-)$-subdivision.
		See Figure~\ref{fig:F123} for an illustration where a blue (or solid) $x_ix_j$-connection presents a positive path some of which could be of length $0$, red (or dashed) connections each shows a negative path, thus each of length at least $1$. We first claim that $C_{F_1}$ and $C_{F_3}$ have no common edge. Otherwise, a common edge $e'$ together with $e_1$ and $e_2$ forms an edge-cut, and by switching at this edge-cut we have a signature with only $2$ negative edges.

		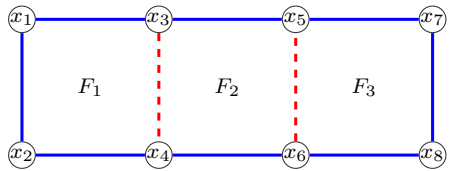
\begin{figure}[htbp]
			\centering
			\begin{minipage}[t]{.4\textwidth}
				\centering
				\begin{tikzpicture}
					[scale=.3]
					\draw(9, -3) node[circle, draw=white!80, inner sep=0mm, minimum size=3.3mm] (F1){\scriptsize $F_1$};
					
					\draw(15, -3) node[circle, draw=white!80, inner sep=0mm, minimum size=3.3mm] (F2){\scriptsize $F_2$};
					
					\draw(21, -3) node[circle, draw=white!80, inner sep=0mm, minimum size=3.3mm] (F3){\scriptsize $F_3$};
					
					\foreach \i/\j in {1/1,2/3,3/5,4/7}
					{
						\draw(6*\i, 0) node[circle, draw=black!80, inner sep=0mm, minimum size=3.3mm] (x\j){\scriptsize $x_{\j}$};
					}
					
					\foreach \i/\j in {1/2,2/4,3/6,4/8}
					{
						\draw(6*\i, -6) node[circle, draw=black!80, inner sep=0mm, minimum size=3.3mm] (x\j){\scriptsize $x_{\j}$};
					}
					
					\foreach \i/\j in {1/3,3/5,5/7,7/8,8/6,6/4,4/2,2/1}
					{
						\draw  [bend left=18, line width=0.4mm, blue] (x\i) -- (x\j);
					}
					
					\foreach \i/\j in {3/4,6/5}
					{
						\draw  [bend left=18, dashed, line width=0.4mm, red] (x\i) -- (x\j);
					}	
				\end{tikzpicture}
				\caption{$F_1, F_2$ and $F_3$}
				\label{fig:F123}
			\end{minipage}\hfil
		\end{figure}

		Let $C_{F_4}$ and $C_{F_5}$ be the two negative facial cycles of $(G, \sigma)$ such that $e_3\in E(C_{F_4}\cap C_{F_5})$. Observe that each of $C_{F_4}$ and $C_{F_5}$ must share at least one edge with either $C_{F_1}$ or $C_{F_3}$. Otherwise, we would have a set of three edge-disjoint negative cycles, by Observation~\ref{obs:cycleunion} contradicting the assumption that $(G, \sigma)$ is non-decomposable. We now consider the following two cases. 		
		
		\begin{figure}[htbp]
			\centering
			\begin{minipage}[t]{.4\textwidth}
				\centering
				\begin{tikzpicture}
					[scale=.24]
					\draw(9, -3) node[circle, draw=white!80, inner sep=0mm, minimum size=3.3mm] (F1){\scriptsize $F_1$};
					
					\draw(15, -3) node[circle, draw=white!80, inner sep=0mm, minimum size=3.3mm] (F2){\scriptsize $F_2$};
					
					\draw(21, -3) node[circle, draw=white!80, inner sep=0mm, minimum size=3.3mm] (F3){\scriptsize $F_3$};
					
					\foreach \i/\j in {1/1,2/3,3/5,4/7}
					{
						\draw(6*\i, 0) node[circle, draw=black!80, inner sep=0mm, minimum size=3.3mm] (x\j){\scriptsize $x_{\j}$};
					}
					
					\foreach \i/\j in {1/2,2/4,3/6,4/8}
					{
						\draw(6*\i, -6) node[circle, draw=black!80, inner sep=0mm, minimum size=3.3mm] (x\j){\scriptsize $x_{\j}$};
					}
					
					\foreach \i/\j in {1/3,3/5,5/7,8/6,6/4,4/2,2/1}
					{
						\draw  [line width=0.4mm, blue] (x\i) -- (x\j);
					}
					
					\draw (9,3) node[circle, draw=black!80, inner sep=0mm, minimum size=3.3mm] (x9){\scriptsize $x_{9}$};
					\draw (21,3) node[circle, draw=black!80, inner sep=0mm, minimum size=3.3mm] (x0){\scriptsize $x_{0}$};
					
					\draw  [line width=0.4mm, blue] (9,0) -- (x9);
					\draw  [line width=0.4mm, blue] (x0) -- (21,0);

                    \draw (24,-2) node[circle, draw=black!80, inner sep=0mm, minimum size=2.4mm] (y){};
                    \draw (24,-4) node[circle, draw=black!80, inner sep=0mm, minimum size=2.4mm] (z){};
                    \draw  [line width=0.4mm, blue] (x7) --(y)--(z)-- (x8);
                    
					\draw  [line width=0.4mm, blue]  (x0).. controls (24,4) and (28,2.6) .. (y);
					\draw  [line width=0.4mm, blue] (x9) .. controls (24,10) and (34,5) .. (z);
					
					\draw(15, 1.5) node[circle, draw=white!80, inner sep=0mm, minimum size=3.3mm] (F4){\scriptsize $F_4$};
					\draw(22,4.4) node[circle, draw=white!80, inner sep=0mm, minimum size=3.3mm] (F5){\scriptsize $F_5$};
					\draw[dashed, line width=0.4mm, red] (x0)-- (x9);
					
					\foreach \i/\j in {3/4,6/5}
					{
						\draw[dashed, line width=0.4mm, red] (x\i) -- (x\j);
					}	
				\end{tikzpicture}
				\caption{Case ($1$)}
				\label{fig:Case1}
			\end{minipage}
			\begin{minipage}[t]{.4\textwidth}
				\centering
				\begin{tikzpicture}
					[scale=.24]
					\draw(9, -3) node[circle, draw=white!80, inner sep=0mm, minimum size=3.3mm] (F1){\scriptsize $F_1$};
					
					\draw(15, -3) node[circle, draw=white!80, inner sep=0mm, minimum size=3.3mm] (F2){\scriptsize $F_2$};
					
					\draw(21, -3) node[circle, draw=white!80, inner sep=0mm, minimum size=3.3mm] (F3){\scriptsize $F_3$};
					
					\foreach \i/\j in {1/1,2/3,3/5,4/7}
					{
						\draw(6*\i, 0) node[circle, draw=black!80, inner sep=0mm, minimum size=3.3mm] (x\j){\scriptsize $x_{\j}$};
					}
					
					\foreach \i/\j in {1/2,2/4,3/6,4/8}
					{
						\draw(6*\i, -6) node[circle, draw=black!80, inner sep=0mm, minimum size=3.3mm] (x\j){\scriptsize $x_{\j}$};
					}
					
					\foreach \i/\j in {1/3,3/5,5/7,7/8,8/6,6/4,4/2}
					{
						\draw  [bend left=18, line width=0.4mm, blue] (x\i) -- (x\j);
					}
					
					\foreach \i/\j in {3/4,6/5}
					{
						\draw  [bend left=18, dashed, line width=0.4mm, red] (x\i) -- (x\j);
					}	
					
					\foreach \i/\j in {3/4,6/5}
					{
						\draw[dashed, line width=0.4mm, red] (x\i) -- (x\j);
					}	
					
					\draw (9,3) node[circle, draw=black!80, inner sep=0mm, minimum size=3.3mm] (x9){\scriptsize $x_{9}$};
					\draw (15,3) node[circle, draw=black!80, inner sep=0mm, minimum size=3.3mm] (x0){\scriptsize $x_{0}$};
					\draw (6,-2) node[circle, draw=black!80, inner sep=0mm, minimum size=2.4mm] (y){};
					\draw (6,-4) node[circle, draw=black!80, inner sep=0mm, minimum size=2.4mm] (z){};
					\draw  [line width=0.4mm, blue] (9,0) -- (x9);
					\draw  [line width=0.4mm, blue] (x0) -- (15,0);
					\draw(12, 1.5) node[circle, draw=white!80, inner sep=0mm, minimum size=3.3mm] (F4){\scriptsize $F_4$};
					\draw[dashed, line width=0.4mm, red] (x9)-- (x0);
					\draw  [line width=0.4mm, blue] (x9).. controls (6,4.2) and (2,2) .. (y);
					\draw  [line width=0.4mm, blue] (x0) .. controls (6,10) and (-3,5) .. (z);
					\draw  [line width=0.4mm, blue] (x1) -- (y) -- (z) -- (x2);
					
					\draw(8, 4.4) node[circle, draw=white!80, inner sep=0mm, minimum size=3.3mm] (F5){\scriptsize $F_5$};
				\end{tikzpicture}
				\caption{Case ($2$)}
				\label{fig:Case2}
			\end{minipage}
		\end{figure}

		\medskip
		\noindent
		{\bf Case} (1): $C_{F_4}$ shares a common edge with (at least) one of $C_{F_1}$ and $C_{F_3}$, and $C_{F_5}$ shares a common edge with the other. 
  
        By symmetry, we assume that $C_{F_4}$ shares a common edge with $C_{F_1}$, and hence $C_{F_5}$ shares a common edge with $C_{F_3}$. See Figure~\ref{fig:Case1}. Then there is an edge-cut crossing the faces $F_4,F_1,F_2,F_3,F_5$, and $F_4$ in this order containing two positive edges and three negative edges, a contradiction with $\ell(G,\sigma)=3$.

		\medskip
		\noindent
		{\bf Case} (2): Each of $C_{F_4}$ and $C_{F_5}$ shares a common edge with the same $C_{F_i}$ for $i\in\{1,3\}$.
  
        By symmetry, assume that each of $C_{F_4}$ and $C_{F_5}$ shares a common edge with $C_{F_1}$ but none with $C_{F_3}$. See Figure~\ref{fig:Case2}. Therefore, $C_{F_3}$ is edge-disjoint from the negative facial cycles $C_{F_1}, C_{F_4}$, and $C_{F_5}$. 
		Furthermore, by Lemma~\ref{lem:3NegativeCycles}, $C_{F_1} \cup C_{F_4} \cup C_{F_5}$ contains a critically $2$-frustrated signed graph. Note that such a critically $2$-frustrated signed graph is edge-disjoint from $C_{F_3}$. Since $C_{F_3}$ is a negative facial cycle (i.e., a critically $1$-frustrated signed graph), $(G, \sigma)$ is decomposable, a contradiction.
	\end{proof}
	
	\begin{corollary}
		If $(G, \sigma) \in \mathcal{P}^*(3)$, then $(G, \sigma)$ is simple. Moreover, for each minimum signature $\sigma$, every facial cycle contains exactly one negative edge.
	\end{corollary}
	
	\begin{proof}
	    By Observation~\ref{obs:loops}, there is no loop in $(G, \sigma)$ and no two parallel edges of different signs. 
		If there exist two parallel edges with the same sign, then in some planar embedding of $(G, \sigma)$ they induce a positive facial cycle, contradicting Theorem \ref{thm:NoPositiveFace}. The moreover part is immediate from the fact that there are six facial cycles.
	\end{proof}

	Now we are ready to describe the elements of the class $\mathcal{P}^*(3)$.
	
	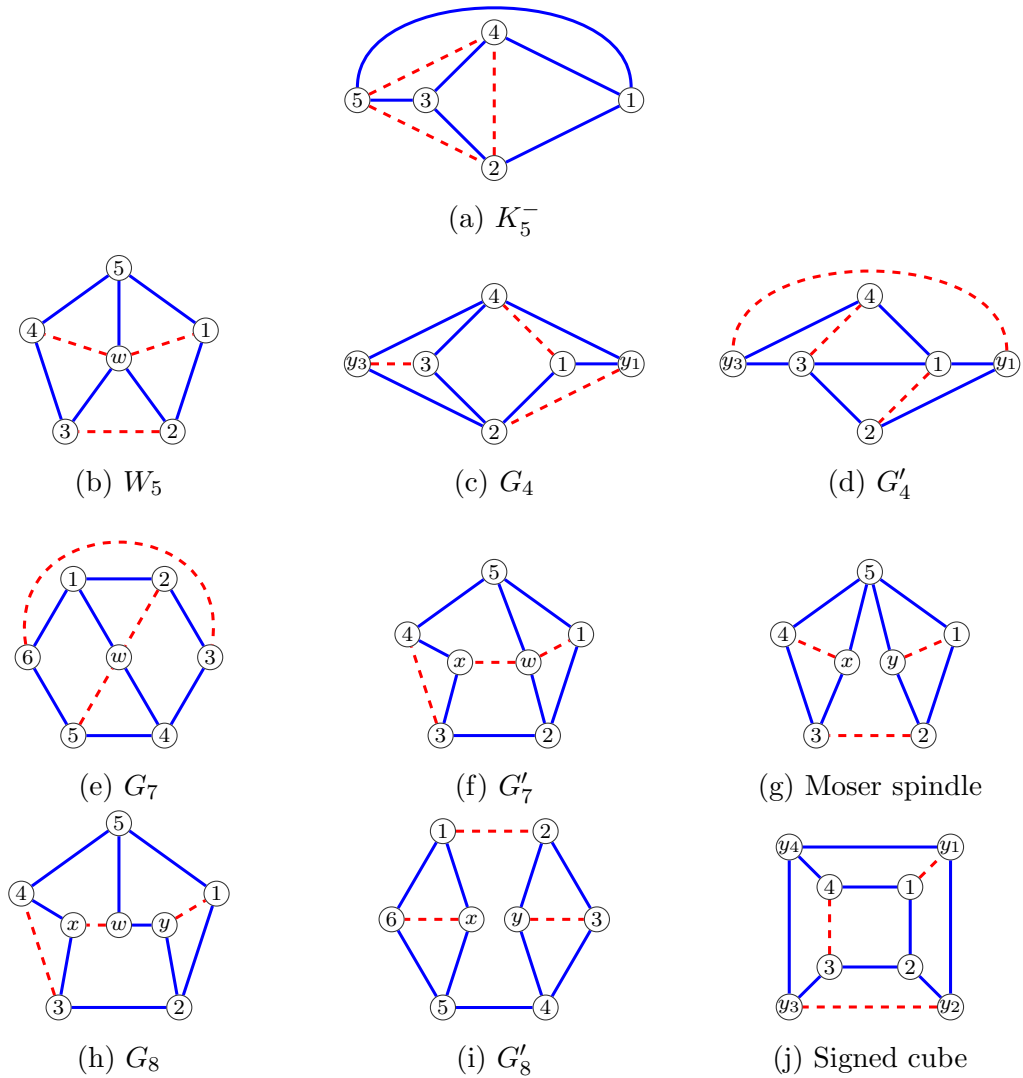
\begin{figure}[htbp]
		\centering 
		\begin{subfigure}[t]{.3\textwidth}
			\centering
			\begin{tikzpicture}
				[scale=.3]
				\foreach \i/\j in {2,3,4}
				{
					\draw[rotate=360-90*(\i)] (0,3) node[circle, draw=black!80, inner sep=0mm, minimum size=3.3mm] (x_{\i}){\scriptsize ${\i}$};}
				
				\draw[rotate=270] (0,6) node[circle, draw=black!80, inner sep=0mm, minimum size=3.3mm] (y_{1}){\scriptsize $1$};
				
				\draw[rotate=90] (0,6) node[circle, draw=black!80, inner sep=0mm, minimum size=3.3mm] (y_{5}){\scriptsize $5$};
				
				\foreach \i/\j in {2/3,3/4}
				{
					\draw  [line width=0.4mm, blue] (x_{\i}) -- (x_{\j});
				}
				\foreach \i/\j in {2/4}
				{
					\draw[dashed, line width=0.4mm, red] (x_{\i}) -- (x_{\j});
				}
				\foreach \i/\j in {1/4,5/3,1/2}
				{
					\draw[line width=0.4mm, blue] (y_{\i}) -- (x_{\j});
				}
				\foreach \i/\j in {5/4,5/2}
				{
					\draw[dashed, line width=0.4mm, red] (y_{\i}) -- (x_{\j});
				}
				
				\draw[line width=0.4mm, blue] (y_{1}) edge[bend right=90] (y_{5});
			\end{tikzpicture}
			\caption{$K_5^-$}
			\label{fig:K_5-}
		\end{subfigure}
		
		\begin{subfigure}[t]{.3\textwidth}
			\centering
			\begin{tikzpicture}
				[scale=.3]
				\draw(0,0) node[circle, draw=black!80, inner sep=0mm, minimum size=3.3mm] (w){\scriptsize $w$};
				
				\foreach \i/\j in {1,2,3,4,5}
				{
					\draw[rotate=360-72*(\i)] (0,4) node[circle, draw=black!80, inner sep=0mm, minimum size=3.3mm] (x_{\i}){\scriptsize ${\i}$};
				}
				
				\foreach \i/\j in {1/2,3/4,4/5,5/1}
				{
					\draw[line width=0.4mm, blue] (x_{\i}) -- (x_{\j});
				}
				\foreach \i/\j in {2/3}
				{
					\draw[dashed, line width=0.4mm, red] (x_{\i}) -- (x_{\j});
				}
				\draw[dashed, line width=0.4mm, red] (w) -- (x_{1});
				\draw[dashed, line width=0.4mm, red] (w) -- (x_{4});
				\draw[line width=0.4mm, blue] (w) -- (x_{2});
				\draw[line width=0.4mm, blue] (w) -- (x_{3});
				\draw[line width=0.4mm, blue] (w) -- (x_{5});
			\end{tikzpicture}
			\caption{$W_5$}
			\label{fig:W5}
		\end{subfigure}
		\begin{subfigure}[t]{.3\textwidth}
			\centering
			\begin{tikzpicture}
				[scale=.3]
				\foreach \i/\j in {1,2,3,4}
				{
					\draw[rotate=360-90*(\i)] (0,3) node[circle, draw=black!80, inner sep=0mm, minimum size=3.3mm] (x_{\i}){\scriptsize ${\i}$};
				}
				
				\foreach \i/\j in {1,3}
				{
					\draw[rotate=360-90*(\i)] (0,6) node[circle, draw=black!80, inner sep=0mm, minimum size=3.3mm] (y_{\i}){\scriptsize $y_{\i}$};
				}
				
				\foreach \i/\j in {1/2,2/3,3/4}
				{
					\draw  [line width=0.4mm, blue] (x_{\i}) -- (x_{\j});
				}
				
				\foreach \i/\j in {1/4}
				{
					\draw  [dashed, line width=0.4mm, red] (x_{\i}) -- (x_{\j});
				}
				
				\foreach \i/\j in {3/2,3/4,1/1,1/4}
				{
					\draw  [line width=0.4mm, blue] (y_{\i}) -- (x_{\j});
				}
				\foreach \i/\j in {3/3,1/2}
				{
					\draw  [dashed, line width=0.4mm, red] (y_{\i}) -- (x_{\j});
				}
			\end{tikzpicture}
			\caption{$G_4$}
			\label{fig:2K_4}
		\end{subfigure}
		\begin{subfigure}[t]{.3\textwidth}
			\centering
			\begin{tikzpicture}
				[scale=.3]
				\foreach \i/\j in {1,2,3,4}
				{
					\draw[rotate=360-90*(\i)] (0,3) node[circle, draw=black!80, inner sep=0mm, minimum size=3.3mm] (x_{\i}){\scriptsize ${\i}$};
				}
				
				\foreach \i/\j in {1,3}
				{
					\draw[rotate=360-90*(\i)] (0,6) node[circle, draw=black!80, inner sep=0mm, minimum size=3.3mm] (y_{\i}){\scriptsize $y_{\i}$};
				}
				
				\foreach \i/\j in {2/3,1/4,1/3}
				{
					\draw  [line width=0.4mm, blue] (x_{\i}) -- (x_{\j});
				}
				
				\foreach \i/\j in {3/4,1/1,1/2, 3/3}
				{
					\draw  [line width=0.4mm, blue] (y_{\i}) -- (x_{\j});
				}
				\foreach \i/\j in {1/2,3/4}
				{
					\draw  [dashed, line width=0.4mm, red] (x_{\i}) -- (x_{\j});
				}
				
				\draw[dashed, line width=0.4mm, red] (y_{1}) edge[bend right=90] (y_{3});
			\end{tikzpicture}
			\caption{$G'_4$}
			\label{fig:G4}
		\end{subfigure}

		\begin{subfigure}[t]{.3\textwidth}
			\centering
			\begin{tikzpicture}
				[scale=.3]
				\draw(0,0) node[circle, draw=black!80, inner sep=0mm, minimum size=3.3mm] (w){\scriptsize $w$};
				
				\foreach \i/\j in {1,2,3,4,5,6}
				{
					\draw[rotate=360-60*(\i)+90] (0,4) node[circle, draw=black!80, inner sep=0mm, minimum size=3.3mm] (x_{\i}){\scriptsize ${\i}$};
				}
				
				\foreach \i/\j in {1/2,2/3,3/4,4/5,5/6,6/1}
				{
					\draw[line width=0.4mm, blue] (x_{\i}) -- (x_{\j});
				}
				
				\foreach \i in {2,5}
				{
					\draw[dashed, line width=0.4mm, red] (x_{\i}) -- (w);
				}
				
				\foreach \i in {1,4}
				{
					\draw[line width=0.4mm, blue] (x_{\i}) -- (w);
				}
				
				\draw[dashed, line width=0.4mm, red] (x_{6}) .. controls (-5, 6.6) and (5, 6.6) .. (x_{3});
			\end{tikzpicture}
			\caption{$G_7$}
			\label{fig:G7_2}
		\end{subfigure}
		\begin{subfigure}[t]{.3\textwidth}
			\centering
			\begin{tikzpicture}
				[scale=.3]
				\draw(1.5,0) node[circle, draw=black!80, inner sep=0mm, minimum size=3.3mm] (w){\scriptsize $w$};
				\draw(-1.5,0) node[circle, draw=black!80, inner sep=0mm, minimum size=3.3mm] (x){\scriptsize $x$};
				
				\foreach \i/\j in {1,2,3,4,5}
				{
					\draw[rotate=360-72*(\i)] (0,4) node[circle, draw=black!80, inner sep=0mm, minimum size=3.3mm] (x_{\i}){\scriptsize ${\i}$};
				}
				
				\foreach \i/\j in {1/2,2/3,4/5,5/1}
				{
					\draw[line width=0.4mm, blue] (x_{\i}) -- (x_{\j});
				}
				\foreach \i/\j in {3/4}
				{
					\draw[dashed, line width=0.4mm, red] (x_{\i}) -- (x_{\j});
				}
				\draw[line width=0.4mm, blue] (w) -- (x_{5});
				\draw[dashed, line width=0.4mm, red] (w) -- (x);
				
				\draw[line width=0.4mm, blue] (x) -- (x_{4});
				\draw[line width=0.4mm, blue] (x) -- (x_{3});
				
				\draw[line width=0.4mm, blue] (w) -- (x_{2});
				\draw[dashed, line width=0.4mm, red] (w) -- (x_{1});
			\end{tikzpicture}
			\caption{$G'_7$}
			\label{fig:G7}
		\end{subfigure}
		\begin{subfigure}[t]{.3\textwidth}
			\centering
			\begin{tikzpicture}
				[scale=.3]
				\draw(-1,0) node[circle, draw=black!80, inner sep=0mm, minimum size=3.3mm] (x){\scriptsize $x$};
				\draw(1,0) node[circle, draw=black!80, inner sep=0mm, minimum size=3.3mm] (y){\scriptsize $y$};
				
				\foreach \i/\j in {1,2,3,4,5}
				{
					\draw[rotate=360-72*(\i)] (0,4) node[circle, draw=black!80, inner sep=0mm, minimum size=3.3mm] (x_{\i}){\scriptsize ${\i}$};
				}
				
				\foreach \i/\j in {1/2,3/4,4/5,5/1}
				{
					\draw[line width=0.4mm, blue] (x_{\i}) -- (x_{\j});
				}
				\foreach \i/\j in {2/3}
				{
					\draw[dashed, line width=0.4mm, red] (x_{\i}) -- (x_{\j});
				}
				\draw[line width=0.4mm, blue] (x) -- (x_{5});
				\draw[line width=0.4mm, blue] (x) -- (x_{3});
				\draw[dashed, line width=0.4mm, red] (x) -- (x_{4});
				
				\draw[line width=0.4mm, blue] (y) -- (x_{5});
				\draw[line width=0.4mm, blue] (y) -- (x_{2});
				\draw[dashed, line width=0.4mm, red] (y) -- (x_{1});
			\end{tikzpicture}
			\caption{Moser spindle}
			\label{fig:Moser}
		\end{subfigure}

		\begin{subfigure}[t]{.3\textwidth}
			\centering
			\begin{tikzpicture}
				[scale=.3]
				\draw(0,0) node[circle, draw=black!80, inner sep=0mm, minimum size=3.3mm] (w){\scriptsize $w$};
				\draw(-2,0) node[circle, draw=black!80, inner sep=0mm, minimum size=3.3mm] (x){\scriptsize $x$};
				\draw(2,0) node[circle, draw=black!80, inner sep=0mm, minimum size=3.3mm] (y){\scriptsize $y$};
				
				\foreach \i/\j in {1,2,3,4,5}
				{
					\draw[rotate=360-72*(\i)] (0,4.5) node[circle, draw=black!80, inner sep=0mm, minimum size=3.3mm] (x_{\i}){\scriptsize ${\i}$};
				}
				
				\foreach \i/\j in {1/2,2/3,4/5,5/1}
				{
					\draw[line width=0.4mm, blue] (x_{\i}) -- (x_{\j});
				}
				\foreach \i/\j in {3/4}
				{
					\draw[dashed, line width=0.4mm, red] (x_{\i}) -- (x_{\j});
				}
				\draw[line width=0.4mm, blue] (w) -- (y);
				\draw[line width=0.4mm, blue] (w) -- (x_{5});
				\draw[dashed, line width=0.4mm, red] (w) -- (x);
				
				\draw[line width=0.4mm, blue] (x) -- (x_{4});
				\draw[line width=0.4mm, blue] (x) -- (x_{3});
				
				\draw[line width=0.4mm, blue] (y) -- (x_{2});
				\draw[dashed, line width=0.4mm, red] (y) -- (x_{1});
			\end{tikzpicture}
			\caption{$G_8$}
			\label{fig:G8_2}
		\end{subfigure}
		\begin{subfigure}[t]{.3\textwidth}
			\centering
			\begin{tikzpicture}
				[scale=.3]
				\draw(-1,0) node[circle, draw=black!80, inner sep=0mm, minimum size=3.3mm] (x){\scriptsize $x$};
				\draw(1,0) node[circle, draw=black!80, inner sep=0mm, minimum size=3.3mm] (y){\scriptsize $y$};
				
				\foreach \i/\j in {1,2,3,4,5,6}
				{
					\draw[rotate=360-60*(\i)+90] (0,4.5) node[circle, draw=black!80, inner sep=0mm, minimum size=3.3mm] (x_{\i}){\scriptsize ${\i}$};
				}
				
				\foreach \i/\j in {2/3,3/4,4/5,5/6,6/1}
				{
					\draw[line width=0.4mm, blue] (x_{\i}) -- (x_{\j});
				}
				\draw[dashed, line width=0.4mm, red] (x_{1}) -- (x_{2});
				
				\draw[line width=0.4mm, blue] (x) -- (x_{1});
				\draw[line width=0.4mm, blue] (x) -- (x_{5});
				\draw[dashed, line width=0.4mm, red] (x) -- (x_{6});
				
				\draw[line width=0.4mm, blue] (y) -- (x_{2});
				\draw[line width=0.4mm, blue] (y) -- (x_{4});
				\draw[dashed, line width=0.4mm, red] (y) -- (x_{3});
			\end{tikzpicture}
			\caption{$G_8'$}
			\label{fig:G8_3}
		\end{subfigure}
		\begin{subfigure}[t]{.3\textwidth}
			\centering
			\begin{tikzpicture}
				[scale=.3]
				\foreach \i/\j in {1,2,3,4}
				{
					\draw[rotate=360-90*(\i)+45] (0,2.5) node[circle, draw=black!80, inner sep=0mm, minimum size=3.3mm] (x_{\i}){\scriptsize ${\i}$};
				}
				
				\foreach \i/\j in {1,2,3,4}
				{
					\draw[rotate=360-90*(\i)+45] (0,5) node[circle, draw=black!80, inner sep=0mm, minimum size=3.3mm] (y_{\i}){\scriptsize $y_{\i}$};
				}
				
				\foreach \i/\j in {4/1,1/2,2/3}
				{
					\draw  [line width=0.4mm, blue] (x_{\i}) -- (x_{\j});
				}
				\foreach \i/\j in {3/4}
				{
					\draw  [dashed, line width=0.4mm, red] (x_{\i}) -- (x_{\j});
				}
				\foreach \i/\j in {3/4,4/1,1/2}
				{
					\draw  [line width=0.4mm, blue] (y_{\i}) -- (y_{\j});
				}
				\foreach \i/\j in {2/3}
				{
					\draw  [dashed, line width=0.4mm, red] (y_{\i}) -- (y_{\j});
				}
				\foreach \i/\j in {2/2,3/3,4/4}
				{
					\draw  [line width=0.4mm, blue] (x_{\i}) -- (y_{\j});
				}
				\foreach \i/\j in {1/1}
				{
					\draw  [dashed, line width=0.4mm, red] (x_{\i}) -- (y_{\j});
				}
			\end{tikzpicture}
			\caption{Signed cube}
			\label{fig:SignedCube}
		\end{subfigure}
		\caption{The class $\mathcal{P}^*(3)$}
		\label{fig:P*3}
	\end{figure}

	\begin{theorem}
		The class $\mathcal{P}^*(3)$ consists of ten signed graphs, depicted in Figure~\ref{fig:P*3}.
	\end{theorem}
	
	\begin{proof}
		Let $(G, \sigma)\in \mathcal{P}^*(3)$ with a planar embedding. By Theorem~\ref{thm:NoPositiveFace}, in $(G, \sigma)$ there are six facial cycles all of which are negative. This determines the signature up to a switching. So it remains to classify the underlying graphs $G$.
		Let $n=|V(G)|$, $m=|E(G)|$, and $f=|F(G)|$ where $F(G)$ is the set of facial cycles of $G$. Note that $f=6$ by Theorem~\ref{thm:NoPositiveFace}. By Euler's formula and the fact that $\delta(G)\geq 3$, we have that $n-\frac{3}{2}n+6\geq 2$. Hence, every irreducible non-decomposable critically $3$-frustrated signed planar graph contains at most $8$ vertices. Note that any simple signed graph on at most four vertices has its frustration index at most $2$, thus $n\geq 5$. Depending on the values of $n$ we consider four cases. Noting that in each case $G$ has $6$ faces, the number of edges is determined by Euler's formula.
		
		\begin{itemize}
			\item $n=5,~ m=9$: The underlying graph is $K_5^-$ as it has only one edge less than $K_5$. This graph has a unique planar embedding and in $(G,\sigma)$ all facial cycles must be negative. In Figure~\ref{fig:K_5-} one such signature is presented.
			
			\item $n=6,~ m=10$: Either $G$ consists of one $5$-vertex and four $3$-vertices or it consists of four $3$-vertices and two $4$-vertices. In the first case, $G$ is isomorphic to $W_5$, see Figure~\ref{fig:W5}. In the second case, we consider two subcases: (1) The two $4$-vertices are not adjacent. In this case, these two $4$-vertices are both adjacent to all the remaining vertices, moreover, there are only two edges induced by the four $3$-vertices. See Figure~\ref{fig:2K_4}. (2) The two $4$-vertices are adjacent. In this case, the two $4$-vertices share at most two common neighbors. Otherwise, a $K_{3,3}$ is forced by just counting degrees, contradicting planarity. The degree conditions then lead to the unique example of Figure~\ref{fig:G4}.
			
			\item $n=7,~ m=11$: $G$ consists of one $4$-vertex and six $3$-vertices. We consider the graph $G_1$ obtained from $G$ by removing the $4$-vertex. Note that $G_1$ consists of two $3$-vertices and four $2$-vertices, and moreover, $G$ is planar and there is a planar embedding such that the four $2$-vertices are in a facial cycle. Then one of the following must be the case for $G_1$: (1) It consists of two $4$-cycles sharing one edge, see Figure~\ref{fig:G7_2}; (2) It consists of one $5$-cycle sharing one edge with a triangle, see Figure~\ref{fig:G7}; (3) It consists of two triangles connected by an edge, see Figure~\ref{fig:Moser}.
			
			\item $n=8,~ m=12$: There is a total of five cubic 2-connected graphs, see for example \cite{BCCS77}. Of these, we have one Wagner graph which is not planar, and one obtained from $K_{3,3}$ by blowing up a vertex to a triangle. The other three form the full list of cubic $2$-connected simple planar graphs on $8$ vertices. They are depicted in Figures~\ref{fig:G8_2}, \ref{fig:G8_3}, and \ref{fig:SignedCube}.
		\end{itemize}
	
	To complete the proof, we need to verify that each signed graph in the list is critically $3$-frustrated. That is to say, removing any edge in any of these signed graphs the remaining subgraph has its frustration index being at most $2$. To see this, we note that each of these ten graphs is $2$-edge-connected and each has only six facial cycles all of which are negative. Thus once an edge is removed, we have five facial cycles, one of which (the new facial cycle) is positive and the other four are negative. It can then be readily verified that in each case these four negative facial cycles can be covered with $2$ edges. By Lemma~\ref{lem:NegativeCycleCover}, we are done.
	\end{proof}

 \section[Constructions]{Constructions of critically $k$-frustrated signed graphs}\label{sec:Examples}

In this section, we first introduce a method to build critically frustrated signed graphs from two given critically frustrated signed graphs, and show that it preserves the property of being non-decomposable and irreducible. Secondly, we build an infinite family of decomposable irreducible critically $3$-frustrated signed graphs. In particular, it implies that the condition of being non-decomposable in Conjectures~\ref{conj:Lk} and ~\ref{conj:Pk} is necessary.

 \subsection{Construction of non-decomposable critically frustrated signed graphs}
   In this subsection, we build signed graphs in $\mathcal{L}^*(k)$ from two given non-decomposable critically frustrated signed graphs, one being $k_1$-frustrated and the other being $k_2$-frustrated such that $k=k_1+k_2-1$.

    \begin{definition}
    	Let $(G_1, \sigma_1)$ and $(G_2, \sigma_2)$ be two signed graphs, and let $xy$ be a negative edge of $(G_1, \sigma_1)$ and $uv$ be a negative edge of $(G_2, \sigma_2)$. We define  $H[(G_1, \sigma_1)_{xy}, ~(G_2, \sigma_2)_{uv}]$ to be the signed graph obtained from disjoint union of $(G_1, \sigma_1)$ and $(G_2, \sigma_2)$ by deleting edges $xy$ and $uv$, and then adding a negative edge $xu$ and a positive edge $yv$. 
    \end{definition}

    \begin{proposition}
    Given integers $k_1, k_2 \geq 2$, let $(G_1, \sigma_1)\in \mathcal{L}^*(k_1)$ and $(G_2, \sigma_2)\in \mathcal{L}^*(k_2)$ be two signed graphs such that $|E^-_{\sigma_1}|=k_1$ and $|E^-_{\sigma_2}|=k_2$. Let $xy$ be a negative edge of $(G_1, \sigma_1)$ and $uv$ be a negative edge of $(G_2, \sigma_2)$. Then $H[(G_1, \sigma_1)_{xy}, ~(G_2, \sigma_2)_{uv}] \in \mathcal{L}^*(k_1+k_2-1)$.
    \end{proposition}

    \begin{proof}
        Let $\sigma$ be the signature of $H[(G_1, \sigma_1)_{xy}, ~(G_2, \sigma_2)_{uv}]$ and note that it has $k_1 + k_2 -1$ negative edges. We first verify that $\sigma$ is a minimum signature by showing that there is no edge-cut with more negative edges than positive ones.
        Suppose to the contrary that there exists an edge-cut of $H[(G_1, \sigma_1)_{xy}, ~(G_2, \sigma_2)_{uv}]$ with more negative edges than positive ones. As $\sigma_1$ (resp. $\sigma_2)$ is a minimum signature of $(G_1, \sigma_1)$ (resp. $(G_2, \sigma_2)$), such an edge-cut, say $\partial(X)$, must contain the new negative edge $xu$. The vertices $x$ and $y$ are not separated by $\partial(X)$ because otherwise in the restriction of $\partial(X)$ to $(G_1, \sigma_1)$ we will find a contradiction. Similarly, $u$ and $v$ are not separated by $\partial(X)$. Then $yv$ is also an edge of $\partial(X)$. However, in this case in one of the restrictions of $\partial(X)$ to  $(G_1, \sigma_1)$ and $(G_2, \sigma_2)$ we find a contradiction.

        Next we show that $H[(G_1, \sigma_1)_{xy}, ~(G_2, \sigma_2)_{uv}]$ is critically  frustrated. By Theorem~\ref{thm:characterizations}, it suffices to prove that each positive edge of $H[(G_1, \sigma_1)_{xy}, ~(G_2, \sigma_2)_{uv}]$ belongs to an equilibrated cut.
        For any positive edge $e$ of $E(G_1, \sigma_1)$, the equilibrated cut of $(G_1, \sigma_1)$ containing $e$ is also an equilibrated cut of $H[(G_1, \sigma_1)_{xy}, ~(G_2, \sigma_2)_{uv}]$ by replacing $xy$ with $xu$ if needed. The same argument holds for positive edges of $(G_2, \sigma_2)$. For the new positive edge $yv$, $\partial(V(G_1))$ is the required equilibrated cut. 
        Note that $H[(G_1, \sigma_1)_{xy}, ~(G_2, \sigma_2)_{uv}]$ is irreducible because it has no vertex with exactly two neighbors.

        \smallskip
        It remains to show that $H[(G_1, \sigma_1)_{xy}, ~(G_2, \sigma_2)_{uv}]$ is not decomposable. Assume to the contrary that it is and suppose there is a $(r_1,\dots, r_t)$-decomposition ($r_1+\cdots +r_t=k_1+k_2-1$) into signed subgraphs $\hat{H}'_1,\dots, \hat{H}'_t$. We may furthermore assume that each $\hat{H}'_i$ is connected. Then they must be 2-connected because a critically frustrated signed graph cannot have a bridge. Thus one of the $\hat{H}'_i$'s, say $\hat{H}'_1$, should contain both $xu$ and $yv$. Each of the others then should be a subgraph of either $(G_1, \sigma_1)$ or $(G_2, \sigma_2)$. Without loss of generality, we assume $\hat{H}'_2$ is a subgraph of $(G_2, \sigma_2)$. Let $(H_2,\sigma)=\hat{H}'_2$, and let $(H_1, \sigma)$ be the signed subgraph obtained from putting all other $\hat{H}'_i$'s (that is  $H[(G_1, \sigma_1)_{xy}, ~(G_2, \sigma_2)_{uv}]-(H_2, \sigma)$). This gives us an $(l_1,l_2)$-decomposition where $l_1=k_1+k_2-1-r_2$ and $l_2=r_2$. 
        
        Observe that $l_2\leq k_2-1$, because $uv$ is not an edge of the critically $l_2$-frustrated signed graph $(H_2, \sigma)$ which is a subgraph of the critically $k_2$-frustrated signed graph $(G_2,\sigma_2)$.
        Let $(H', \sigma_2)$ be the signed subgraph of $(G_2, \sigma_2)$ by removing all edges of $H_2$ (recall that $uv$ is a negative edge of this signed subgraph). Observe that $\ell(H', \sigma_2)\leq k_2-l_2 $, but moreover if $\ell(H', \sigma_2)= k_2-l_2 $ then by Observation~\ref{obs:cycleunion} $(G_2, \sigma_2)$ is $(l_2, k_2-l_2)$-decomposable, a contradiction. Thus $\ell(H', \sigma_2)\leq k_2-l_2-1$. Thus there exists a switching-equivalent signature $\pi$ of $\sigma_2$ such that $|E^-_{\pi}(H')|=k_2-l_2-1$. Assume $\pi$ is obtained by switching at a set $X$ of vertices of $G_2$. 
         
        We consider two cases based on whether $uv\in E^-_\pi$. If $uv\in E^-_\pi$, then $X$ contains either both of $u$ and $v$ or none of them. We now consider a switching at the subset $X$ of the vertices of $(H_1, \sigma)$. This switching does not change the signs of the edges in $(G_1, \sigma_1)$ part, thus there remain $k_1-1$ negative edges in this part, noting that $xy$ is not an edge in $E(H_1\cap G_1)$. On $\{xu, yv\}$ there would remain one negative edge. And on $(H'-uv, \pi)$ we have $k_2-l_2-1$ negative edges. Altogether we have $k_1+k_2-l_2-2$ negative edges in this switching of $(H_1, \sigma)$, contradicting the fact that its frustration index is $k_1+k_2-l_2-1$.         
        If $uv\not\in E^-_\pi$, then $X$ contains exactly one of $u$ or $v$, by symmetry of switching on $X$ or $X^{c}$, we may assume $u\in X$. As in the previous case we consider a switching at the subset $X$ of the vertices of $(H_1, \sigma)$. Since $u\in X$ and $v\not\in X$, both $xu$ and $yv$ are positive edges after this switching. A similar calculation as before then counts the number of negative edges in this switched signed graph to be $k_1+k_2-l_2-2$, which leads to the same contradiction.      
    \end{proof}

	\subsection[Infinite]{An infinite family of critically $3$-frustrated signed graphs}

	As mentioned before, the family of critically $2$-frustrated signed graphs consists of $(K_4,-)$-subdivisions and edge-disjoint union of two negative cycles. If we furthermore require that they are irreducible, then there are only three such signed graphs: $(K_{4}, -)$, two disjoint negative loops, and two negative loops on the same vertex. In other words, the set $\mathcal{L}(2)$ of irreducible critically $2$-frustrated signed graphs consists of only three elements even without the added assumption of being non-decomposable. However, that is not the case for critically $k$-frustrated signed graphs for $k\geq 3$. In this subsection, we show the next result.  

    \begin{theorem}\label{thm:Infinite}
    The set $\mathcal{L}(3)$ contains infinitely many irreducible critically $3$-frustrated signed graphs.
    \end{theorem}

  By adding a number of negative loops to the signed graphs of $\mathcal{L}(3)$, one gets examples for any $k$ as long as $k\geq 3$.

 \begin{corollary}
     The set $\mathcal{L}(k)$ is infinite for any positive integer $k\geq 3$.
 \end{corollary}
 
In order to prove our statements, we first define a sequence of signed graphs as follows: Let $\hat{G}_0$ be the signed graph obtained from $K_4$ on vertices $x, y,z,w$ by first assigning negative signs to $xw$ and $yz$, positive signs to the remaining four edges, and secondly adding a positive edge $xw$ and a negative edge $yz$. See Figure~\ref{fig:father}. Observe that $\hat{G}_0$ can be decomposed into three negative cycles: $xwx$ (2-cycle), $xyzx$ (3-cycle), and $wyzw$ (3-cycle). 
	
	\begin{figure}[ht]
		\centering
       \begin{minipage}[t]{.32\textwidth}
       \centering
		\begin{tikzpicture}[scale=.6]
			\draw (0,0) node[circle, draw=black!80, inner sep=0mm, minimum size=3.5mm, fill=white] (w){\scriptsize $w$};
			\draw (0,4) node[circle, draw=black!80, inner sep=0mm, minimum size=3.5mm, fill=white] (x){\scriptsize $x$};
			\draw (-3.5,-2) node[circle, draw=black!80, inner sep=0mm, minimum size=3.5mm, fill=white] (y){\scriptsize $y$};
			\draw (3.5,-2) node[circle, draw=black!80, inner sep=0mm, minimum size=3.5mm, fill=white] (z){\scriptsize $z$};
			
			\begin{pgfonlayer}{bg}
				\draw[line width=0.4mm, blue] (x) -- (y);				
				\draw[line width=0.4mm, blue, name path=bello] (x) -- (z);				
				\draw[line width=0.4mm, blue] (w) -- (y);
				\draw[line width=0.4mm, blue] (w) -- (z);	
				\draw[dash pattern=on 1mm off .7mm,, line width=0.4mm, red] (y) edge[bend left=15] (z);
				\draw[dash pattern=on 1mm off .7mm,, line width=0.4mm, red] (y) edge[bend left=-15] (z);
				\draw[line width=0.4mm, blue] (x) edge[bend left=15] (w);
				\draw[dash pattern=on 1mm off .7mm,, line width=0.4mm, red] (x) edge[bend left=-15] (w);
			\end{pgfonlayer}
			
		\end{tikzpicture}
		\caption{ $\hat{G}_0$.}
		\label{fig:father}
        \end{minipage}
        \begin{minipage}[t]{.32\textwidth}
            \centering
        \begin{tikzpicture}[scale=.6]
			\draw (0,0) node[circle, draw=black!80, inner sep=0mm, minimum size=3.5mm, fill=white] (w){\scriptsize $w$};
			\draw (0,4) node[circle, draw=black!80, inner sep=0mm, minimum size=3.5mm, fill=white] (x){\scriptsize $x$};
			\draw (-3.5,-2) node[circle, draw=black!80, inner sep=0mm, minimum size=3.5mm, fill=white] (y){\scriptsize $y$};
			\draw (3.5,-2) node[circle, draw=black!80, inner sep=0mm, minimum size=3.5mm, fill=white] (z){\scriptsize $z$};
			
			\begin{pgfonlayer}{bg}
				\draw[line width=0.4mm, blue] (x) -- (y);				
				\draw[line width=0.4mm, blue, name path=bello] (x) -- (z);				
				\draw[line width=0.4mm, blue, name path=bello2] (w) -- (y);
				\draw[line width=0.4mm, blue] (w) -- (z);										
				\draw[dash pattern=on 1mm off .7mm,, line width=0.4mm, red] (x) -- (w);
				\draw[dash pattern=on 1mm off .7mm,, line width=0.4mm, red] (y) edge[bend left=15] (z);
				\draw[dash pattern=on 1mm off .7mm,, line width=0.4mm, red] (y) edge[bend left=-15] (z);
			\end{pgfonlayer}
			
			\path[name path=h1] (w)-- (1,5);
			\path[name intersections={of=h1 and bello,by={i1}}];
			\path[name path=h2] (w)-- (2,5);
			\path[name intersections={of=h2 and bello,by={i2}}];
			\path[name path=h3] (w)-- (4.5,4.5);
			\path[name intersections={of=h3 and bello,by={i3}}];
			
			\path[name path=j1] (x)-- (-3,-2);
			\path[name intersections={of=j1 and bello2,by={i4}}];
			\path[name path=j2] (x)-- (-2.5,-2);
			\path[name intersections={of=j2 and bello2,by={i5}}];
			\path[name path=j3] (x)-- (-1.5,-2);
			\path[name intersections={of=j3 and bello2,by={i6}}];
			
			\draw[rotate=0] (i1)  node[circle,draw=black!80, inner sep=0mm, minimum size=3.2mm, fill=white] (v2){};
			\draw[rotate=0] (i3)  node[circle,draw=black!80, inner sep=0mm, minimum size=3.2mm, fill=white] (v4){};
			\draw[rotate=0] (i5)  node[circle,draw=black!80, inner sep=0mm, minimum size=3.2mm, fill=white] (v1){};
			\draw[rotate=0] (i6)  node[circle,draw=black!80, inner sep=0mm, minimum size=3.2mm, fill=white] (v3){};
			\foreach \i/\j in {1/2,2/3,3/4} {
				\draw[line width=0.4mm, blue] (v\i) edge[bend left=7.6] (v\j);		
			};
			\draw[line width=0.4mm, blue] (v1) -- (x);
			\draw[line width=0.4mm, blue] (v4) -- (w);
		\end{tikzpicture}
		\caption{$\hat{G}_2$}
		\label{fig:G2sigma2}
        \end{minipage}
        \begin{minipage}[t]{.32\textwidth}
       \centering
		\begin{tikzpicture}[scale=.5]
			\draw (0,-2) node[circle, draw=black!80, inner sep=0mm, minimum size=3.2mm, fill=white] (v3){};
			\draw (0,4) node[circle, draw=black!80, inner sep=0mm, minimum size=3.5mm, fill=white] (x){\scriptsize $x$};
			\draw (-3.5,-2) node[circle, draw=black!80, inner sep=0mm, minimum size=3.5mm, fill=white] (y){\scriptsize $y$};
			\draw (5,-2) node[circle, draw=black!80, inner sep=0mm, minimum size=3.5mm, fill=white] (w){\scriptsize $w$};
			\draw (4.3,-4) node[circle, draw=black!80, inner sep=0mm, minimum size=3.5mm, fill=white] (z){\scriptsize $z$};
			\draw (-2,-2) node[circle, draw=black!80, inner sep=0mm, minimum size=3.2mm, fill=white] (v1){};
			
			\begin{pgfonlayer}{bg}
				\draw[line width=0.4mm, blue] (x) -- (y);	
				\draw[line width=0.4mm, blue] (x) -- (v1);							
				\draw[line width=0.4mm, blue, name path=bello] (x) -- (z);				
				\draw[line width=0.4mm, blue, name path=bello2] (w) -- (y);
				\draw[line width=0.4mm, blue] (w) -- (z);										
				\draw[dash pattern=on 1mm off .7mm,, line width=0.4mm, red] (x) edge[bend left=15] (w);
				\draw[dash pattern=on 1mm off .7mm,, line width=0.4mm, red] (y) edge[bend left=-15] (z);
				\draw[dash pattern=on 1mm off .7mm,, line width=0.4mm, red] (y) edge[bend left=-30] (z);
			\end{pgfonlayer}

			\path[name path=h1] (v3)-- (1,5);
			\path[name intersections={of=h1 and bello,by={i1}}];
			\path[name path=h2] (v3)-- (2,5);
			\path[name intersections={of=h2 and bello,by={i2}}];
			\path[name path=h3] (v3)-- (4.5,4.5);
			\path[name intersections={of=h3 and bello,by={i3}}];
			\path[name intersections={of=bello2 and bello,by={i4}}];
			
			\draw[rotate=0] (i1)  node[circle,draw=black!80, inner sep=0mm, minimum size=3.2mm, fill=white] (v2){};
			\draw[rotate=0] (i3)  node[circle,draw=black!80, inner sep=0mm, minimum size=3.2mm, fill=white] (v4){};
			\draw[rotate=0] (i4)  node[circle,draw=black!80, inner sep=0mm, minimum size=3.5mm, fill=white] (v5){\scriptsize $s$};
			
			\foreach \i/\j in {1/2,2/3,3/4} {
				\draw[line width=0.4mm, blue] (v\i) -- (v\j);				
			};
			\draw[line width=0.4mm, blue] (v4) -- (w);

		\end{tikzpicture}
		\caption{$\hat{G}'_2$}
		\label{fig:planarDec}
  \end{minipage}
	\end{figure}
	
	The signed graph $\hat{G}_t$ of the sequence is built from $\hat{G}_0$ as follows. We first introduce $2t$ points by subdividing the positive edge connecting $x$ and $w$, and two sets of $t$ points by subdividing each of $xz$ and $yw$. Then we identify the $2t$ points of the $xw$-path with the $2t$ points, alternating between the points from $xz$ and $wy$. See Figure~\ref{fig:G2sigma2} for the case of $t=2$. 

    \medskip
    \noindent
    \emph{Proof of Theorem~\ref*{thm:Infinite}.}
    We shall prove this claim by showing that $\hat{G}_t\in \mathcal{L}(3)$.
    Observe that subdivisions of each of the three cycles given in decomposition of $\hat{G}_{0}$ gives a decomposition of $\hat{G}_t$. It implies that $\ell (\hat{G}_t)=3$. What remains is to show that $\hat{G}_t$ is irreducible and critically $3$-frustrated.
	
	That $\hat{G}_t$ is irreducible follows from the fact that in a subdivision of a graph, there is always a vertex that has only two distinct neighbors. But there is no such vertex in $\hat{G}_t$. Now we provide a sketch of the proof of $\hat{G}_t$ being critically $3$-frustrated. First, observe that each edge incident with $y$ (or $z$) is in an equilibrated cut $\partial(y)$ (respectively, $\partial(z)$). All other edges are the results of subdivisions (and then identifying some vertices). For an edge $uv$ where $u$ is a vertex on the subdivision of $xz$ and $v$ is a vertex on the subdivision of $yw$, the following six edges form an equilibrated cut: $uv$, the edge on the $xz$-path that forms a triangle with $uv$, the edge on the $yw$-path that forms a triangle with $uv$ and the three negative edges.
    \hfill$\Box$
	
	\smallskip
    In fact, we can modify these signed graphs to get an infinite family of irreducible critically $3$-frustrated signed planar graphs. For each $\hat{G}_t$, we apply the following modification to get $\hat{G}'_t$. First, by modifying the embedding of Figure~\ref{fig:G2sigma2} and putting $w$ on the outside of the $xyz$-triangle, we may have an embedding with one cross which is the crossing of the edge of the $yw$-path incident with $w$ and the edge of the $xz$-path incident with $z$. Then introduce a new vertex, say $s$ at this crossing point to get the planar signed graph $\hat{G}'_t$. See Figure~\ref{fig:planarDec} for a depiction of $\hat{G}'_2$. The only remaining point to verify is that each of the new edges is in an equilibrated cut. Such two cuts are $\partial(\{w,z\})$ and $\partial(\{w,z,s\})$.
    Therefore, we obtain the following result for planar graphs.

    \begin{theorem}
    There exist infinitely many irreducible critically $k$-frustrated planar signed graphs for $k\geq 3$.
    \end{theorem}
	
	We remark that even though the classes $\mathcal{S}^*(3)$ and $\mathcal{P}^*(3)$ are fully described in this work, the full description of the class $\mathcal{L}^*(3)$ is far from clear. In particular, $\mathcal{L}^*(3)\setminus (\mathcal{S}^*(3)\cup \mathcal{P}^*(3))\neq \emptyset$. Two examples of such signed graphs are given in Figure~\ref{fig:K5minor}. The class $\mathcal{L}^*(3)$ is shown to contain finitely many elements in forthcoming work.

	\begin{figure}[htbp]
		\centering
		\begin{minipage}[t]{.4\textwidth}
			\centering
			\begin{tikzpicture}[scale=.34]
				\foreach \i/\j in {1,2,3,4,5}
				{
					\draw[rotate=360-72*(\i)] (0,5) node[circle, draw=black!80, inner sep=0mm, minimum size=3.3mm] (x_{\i}){};
				}
				
				\foreach \i/\j in {1,2,3,4,5}
				{
					\draw[rotate=360-72*(\i)] (0,3) node[circle, draw=black!80, inner sep=0mm, minimum size=3.3mm] (y_{\i}){};
				}
				
				\foreach \i/\j in {1/2,3/4,4/5,5/1}
				{
					\draw  [line width=0.4mm, blue] (x_{\i}) -- (x_{\j});
				} 
				
				\foreach \i/\j in {1/3,3/5,5/2,2/4}
				{
					\draw  [line width=0.4mm, blue] (y_{\i}) -- (y_{\j});
				} 
				
				\foreach \i/\j in {2/3}
				{
					\draw  [dashed, line width=0.4mm, red] (x_{\i}) -- (x_{\j});
				}
				
				\foreach \i/\j in {1/4}
				{
					\draw  [dashed, line width=0.4mm, red] (y_{\i}) -- (y_{\j});
				}
				
				\foreach \i in {1,2,3,4}
				{
					\draw  [line width=0.4mm, blue] (x_{\i}) -- (y_{\i});
				}
				
				\draw[dashed, line width=0.4mm, red] (x_{5}) -- (y_{5});
			\end{tikzpicture}
		\end{minipage}
		\begin{minipage}[t]{.4\textwidth}
			\centering
			\begin{tikzpicture}
				[scale=.26]
				\foreach \i/\j in {1/1,2/3,3/5,4/7}
				{
					\draw(6*\i, 0) node[circle, draw=black!80, inner sep=0mm, minimum size=3.3mm] (x\j){};
				}
				
				\foreach \i/\j in {1/2,2/4,3/6}
				{
					\draw(6*\i, -12) node[circle, draw=black!80, inner sep=0mm, minimum size=3.3mm] (x\j){};
				}
				
				\foreach \i/\j in {1/4,1/6,3/2,3/6,5/2,5/4,7/2,7/4,7/6}
				{
					\draw  [bend left=18, line width=0.4mm, blue] (x\i) -- (x\j);
				}
				
				\foreach \i/\j in {1/2,3/4,5/6}
				{
					\draw  [bend left=18, dashed, line width=0.4mm, red] (x\i) -- (x\j);
				}	
			\end{tikzpicture}
		\end{minipage}
		\caption{Examples in $\mathcal{L}^*(3)$ neither in $\mathcal{S}^*(3)$ nor in $\mathcal{P}^*(3)$}
		\label{fig:K5minor}
	\end{figure}
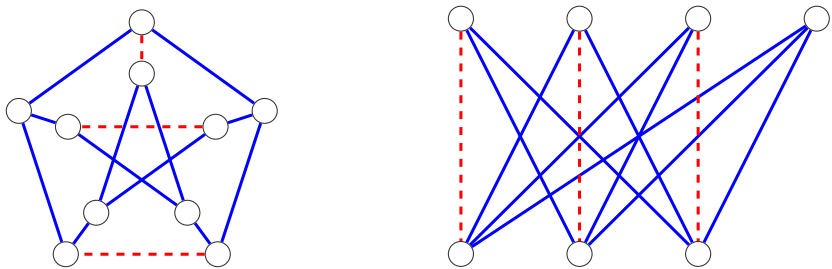

	\section{A support for Conjecture~\ref{conj:degree2k}}\label{sec:degree}
	
	In this section, we show that Conjecture~\ref{conj:degree2k} holds if we add the extra condition that $(G,\sigma)$ has no $(K_5,-)$-minor.
	A signed graph $(H, \pi)$ is a \emph{minor} of $(G,\sigma)$ if it is obtained from $(G, \sigma)$ by a sequence of the following operations: Deleting vertices or edges, contraction of positive edges, and switching.
	
	Our claim concludes from some known results on the frustration index of $(K_5,-)$-minor-free signed graphs.
		
	\begin{theorem}\label{Schrijver2}{\rm \cite{S03}}
		Let $(G, \sigma)$ be an Eulerian signed graph without $(K_5, -)$-minor. Then the
		maximum number of edge-disjoint negative cycles of $(G, \sigma)$ is equal to its frustration index.
	\end{theorem}
	
	A set $ \mathcal{C}$ of negative cycles of $(G, \sigma)$ is said to be a \emph{ $\leq_2$-negative cycle cover} if each edge belongs to at most two cycles of $\mathcal{C}$. If each edge belongs to exactly two cycles in $\mathcal{C}$, then we say that $\mathcal{C}$ is a \emph{negative cycle double cover}. 
	
	Let $(G, \sigma)$ be a $(K_5, -)$-minor-free signed graph. Note that, by doubling each edge with the respective sign, we obtain a new $(K_5, -)$-minor-free signed graph which is Eulerian and whose frustration index equals $2 \ell(G, \sigma)$. By applying Theorem~\ref{Schrijver2} to this new signed graph, we obtain the following result.

	\begin{theorem}\label{Schrijver2Cor}
	Let $(G, \sigma)$ be a $(K_5, -)$-minor-free signed graph and let $\mathcal{C}$ be a $\leq_2$-negative cycle cover of $(G, \sigma)$. Then $\ell(G,\sigma)= \frac{1}{2}\,|\mathcal{C}|$. 
	\end{theorem}
	
	We use the notion of critically frustrated signed graphs to  obtain the following strengthening. 
	
	\begin{theorem}
		Every $(K_5, -)$-minor-free critically $k$-frustrated signed graph has a negative cycle double cover of order $2k$. 
	\end{theorem}
	
	\begin{proof}
		Let $(G, \sigma)$ be a $(K_5, -)$-minor-free signed graph and assume it is critically $k$-frustrated. By Theorem~\ref{Schrijver2Cor} there exists a $\leq_2$-negative cycle cover $\mathcal{C}$ of cardinality $2k$.
		We prove that $\mathcal{C}$ is indeed a negative cycle double cover. That is to say that each edge of $G$ is in  exactly two cycles of $\mathcal{C}$. Assume to the contrary that an edge $e$ is not in two cycles of $\mathcal{C}$, thus it is either in none of them or only in one of them.
		
		First, consider the case when $e$ does not belong to any cycle of $\mathcal{C}$. Then $\ell(G-e, \sigma)\geq k=\frac{1}{2}|\mathcal{C}|$, contradicting the criticality of $(G, \sigma)$.
		
		Next, suppose that $e$ belongs to exactly one cycle of $\mathcal{C}$. By criticality, $\ell(G-e, \sigma)=2k-1$. Hence, by Theorem~\ref{Schrijver2Cor}, each $\leq_2$-negative cycle cover $\mathcal{C}$ of $\ell(G-e, \sigma)$ is of order at most $2k-2$.
		Since $e$ belongs to exactly one cycle of $C \in \mathcal{C}$, the set $\mathcal{C} \setminus \{C\}$ is a $\leq_2$-negative cycle cover of $(G-e, \sigma)$ with $|\mathcal{C} \setminus \{C\}|=2k-1$, a contradiction.
	\end{proof}
	
	As the edges incident with each vertex $v$ belong to at most $2k$ cycles, the Conjecture~\ref{conj:degree2k} is implied when $(G, \sigma)$ has no $(K_5,-)$-minor. 
	
	\begin{corollary}
        Every $(K_5, -)$-minor-free critically $k$-frustrated signed graph $(G, \sigma)$ satisfies $\Delta(G) \leq 2k$.
	\end{corollary}
	
	\medskip
	{\bf Acknowledgment.} This work is partly done when the first author was visiting IRIF. It has also partially supported by the following grants: NRW-Forschungskolleg Gestaltung von flexiblen Arbeitswelten; ANR (France) project HOSIGRA (ANR-17-CE40-0022); European Union's Horizon 2020 research and innovation program under the Marie Sklodowska-Curie grant agreement No 754362.
	
\bibliography{Part1}{}
\addcontentsline{toc}{section}{References}
\bibliographystyle{plain}

		
		
		

		

		


\end{document}